%% file: modular-forms-in-julia.tex
\documentclass[a4paper,10pt,oneside,headsepline]{scrartcl}

\input{packages}
\input{layout}

\input{shortcuts}
\input{shortcuts_this}

\newcommand{\headertitle}{{\normalfont%
  On the Computation of General Vector-valued Modular Forms
}}
\newcommand{\headerauthors}{%
  T.~Magnusson,
  M.~Raum%
}
\title{%
  On the Computation of General\\Vector-valued Modular Forms
}
\author{%
Tobias Magnusson \and
Martin Raum%
\thanks{The author was partially supported by Vetenskapsr\aa det Grant~2015-04139 and~2019-03551.}%
}

\begin{document}

\thispagestyle{scrplain}
\begingroup
\deffootnote[1em]{1.5em}{1em}{\thefootnotemark}
\maketitle
\endgroup

\begin{abstract}
\small
\noindent
{\tbf Abstract:}
We present and discuss an algorithm and its implementation that is capable of directly determining Fourier expansions of any vector-valued modular form of weight at least~$2$ associated with representations whose kernel is a congruence subgroup. It complements two available algorithms that are limited to inductions of Dirichlet characters and to Weil representations, thus covering further applications like Moonshine or Jacobi forms for congruence subgroups. We examine the calculation of invariants in specific representations via techniques from permutation groups, which greatly aids runtime performance. We explain how a generalization of cusp expansions of classical modular forms enters our implementation. After a heuristic consideration of time complexity, we relate the formulation of our algorithm to the two available ones, to highlight the compromises between level of generality and performance that each them makes.
\\[.3\baselineskip]
\noindent
\textsf{\textbf{%
  holomorphic modular forms%
}}%
\noindent
\hspace{0.3em}{\tiny$\blacksquare$}\hspace{0.3em}%
\textsf{\textbf{%
  Fourier expansions%
}}%
\noindent
\hspace{0.3em}{\tiny$\blacksquare$}\hspace{0.3em}%
\textsf{\textbf{%
  cusp expansions%
}}
\\[.2\baselineskip]
\noindent
\textsf{\textbf{%
  MSC Primary:
  11F30%
}}%
\hspace{0.3em}{\tiny$\blacksquare$}\hspace{0.3em}%
\textsf{\textbf{%
  MSC Secondary:
  11F11, 11F50%
}}
\end{abstract}

\Needspace*{4em}
\addcontentsline{toc}{section}{Introduction}
\markright{Introduction}
\lettrine[lines=2,nindent=.2em]{\tbf T}{he} vector-valued case appears late in the chronicles of computing modular forms. The main challenge for several decades was to elucidate the modularity conjecture and special values of~$\rmL$\nbd functions. While they persist, in the past two decades the demand for Fourier expansions of vector-valued modular forms has spiked. Their computational theory, however, has only recently started to unfold.

Via an induction construction (see~\cite{raum-2017}), vector-valued modular forms capture cusp expansions of scalar-valued ones, which lends them significance to classical questions. For instance, the computation of modular forms for congruence subgroups~$\Ga_{\mathrm{ns}}(N)$ of non-split Cartan type was until recently centered around the cusp expansion of modular forms for~$\Ga_1(N^2)$~\cite{mercuri-schoof-2020,assaf-2020-preprint}. Given the role of such groups in Mazur's Program~\cite{mazur-1977}, it was and is an important challenge to determine all cusp expansions of such modular forms, or alternatively to determine the corresponding vector-valued modular forms. Vector-valued modular forms that arise in this setting are associated with, for instance, the induction of Dirichlet characters from~$\Ga_0(N)$ to~$\SL{2}(\ZZ)$.

Work of Borcherds on theta lifts~\cite{borcherds-1998} further accentuates the need for vector-valued modular forms. In light of applications that rely on the concrete knowledge of Fourier expansions~\cite{scheithauer-2006,dittmann-salvati-manni-scheithauer-2021}, more efficient or more general algorithms promise progress in domains like vertex operator algebras or algebraic geometry. Vector-valued modular forms that arise in this setting are associated with Weil representations.

There are presently two algorithms available to compute vector-valued modular forms, one that covers inductions of Dirichlet characters and one that covers Weil representations. But some applications do not fall under the scope of either of them. These applications require vector-valued modular forms associated with more general kinds of representations. Most prominently, Moonshine features modular or Jacobi forms for signed permutation representations~\cite{gaberdiel-persson-ronellenfitsch-volpato-2013,cheng-duncan-harvey-2014}. Also cusp expansions of Jacobi forms yield modular forms for representations that are more generic. We subsume them all under the notion of twisted permutation representations in the sense of our Definition~\ref{def:twisted_perm_representation}.

The need for a more general algorithm to determine vector-valued modular forms has become imminent. Our main theorem asserts the existence of an algorithm that solves this problem if the kernel of the representation is a congruence subgroup. A more precise formulation is available in Theorem~\ref{thm:main_algorithm}.

\begin{maintheorem}
\label{mainthm:main_algorithm}
Fix an integral weight~$k \ge 2$ and a finite-dimensional, complex representation~$\rho$ of a subgroup~$\Ga$ of\/~$\SL{2}(\ZZ)$ whose kernel is a congruence subgroup. Then Algorithm~\ref{alg:main_algorithm} on page~\pageref{alg:main_algorithm} computes the Fourier expansion of a basis of\/~$\rmM_k(\rho)$ with algebraic coefficients up to any given precision.
\end{maintheorem}

The theoretical foundation of Algorithm~\ref{alg:main_algorithm} and Theorem~\ref{mainthm:main_algorithm} is a result on products of Eisenstein series that span spaces of modular forms by Xià and one of the authors~\cite{raum-xia-2020}. The pathway to Algorithm~\ref{alg:main_algorithm} is prepared in Section~\ref{sec:main_algo} and in particular in Theorem~\ref{thm:main_algorithm_preparation}.

In order to gauge the relevance of Theorem~\ref{mainthm:main_algorithm}, we briefly highlight some of the milestones in the computation of vector-valued modular forms. Given the wide availability of Fourier expansions for scalar-valued modular forms, numerical integration around cusps on the real line was a common strategy for an extended period of time. It appears, for example, in an implementation of the~$\rmO(\log(p))$-algorithm to compute the~$p$\thdash\ Fourier coefficient of the Ramanujan~$\Delta$-function~\cite{edixhoven-couveignes-2011}. Like all numerical techniques, it comes with its own challenges connected to the fluctuation of modular forms on low enough horizontal lines. As for algebraic expressions, only cusp expansions of Eisenstein series were generally available~\cite{miyake-1989,diamond-shurman-2005}. Incidentally, Eisenstein series were also the first case for which an algebraic and vector-valued formulation was provided by Bruinier and Kuss~\cite{bruinier-kuss-2001} around~2001. The theory of Jacobi forms and theta blocks~\cite{gritsenko-skoruppa-zagier-2019-preprint} entered the picture around~2010, mostly in work by Poor, Yuen et al.~\cite{gritsenko-poor-yuen-2015,poor-shurman-yuen-2018}, and gave access to vector-valued modular forms associated with some Weil representations. The use of products of vector-valued Eisenstein series to our knowledge was first suggested in~\cite{raum-2017} in 2014, a variation of which was promptly adapted by Cohen~\cite{cohen-2019} to an implementation of classical modular forms~\cite{belabas-cohen-2018}. In 2018 Williams provided, as part of his PhD thesis, the foundation to an algorithm targeting Weil representations~\cite{williams-2018a} and showcased several applications~\cite{williams-2018b,williams-2019}. For historical accuracy, we remark that these vector-valued approaches are predated by the study of products of two scalar-valued Eisenstein series that goes back at least as far as to work of Rankin~\cite{rankin-1952}, and appears in several subsequent articles~\cite{kohnen-zagier-1984,kohnen-martin-2008}. We also record that when relaxing the restriction to products with at most two factors and focusing instead on Eisenstein series of small weight, early results are contained in the work of Borisov--Gunells~\cite{borisov-gunnells-2001,borisov-gunnells-2001b,borisov-gunnells-2003} and the strongest results known to the authors are due to Khuri-Makdisi~\cite{khuri-makdisi-2012}.

We have provided an implementation of Algorithm~\ref{alg:main_algorithm} in the programming language Julia based on the computer algebra packages Nemo/Hecke~\cite{fieker-hart-hofmann-johansson-2017}. We plan to make it also available via the Julia software repository. In the process of creating this implementation, we needed to develop several mathematical tools specific to vector-valued modular forms, which we share in Sections~\ref{sec:decompositions} and~\ref{sec:implementation}. Our algorithm requires the computation of homomorphisms between representations. In Section~\ref{ssec:computation_invariants}, we discuss how this computation can be facilitated by purely group theoretic calculations in the case of twisted permutation representations. In the case of~$\SL{2}(\ZZ)$-representations this reduces the heuristic time complexity from~$\rmO((N^4 \dim(\rho))^\kappa)$ to~$\rmO(N^{3+\epsilon} \dim(\rho)^2)$, where~$\kappa \approx 2.8$, practically, is explained in Section~\ref{ssec:time_complexity}. Further, we employ a deflation--inflation principle that parallels the theory of cusp expansions for classical modular forms, which we discuss in Section~\ref{ssec:infl_defl_torbits} based on Corollary~\ref{cor:diag_tw_perm_taction}. Its impact is considerable, as it reduces heuristically the number of relevant Fourier expansions from from $\rmO(N^{1+\epsilon})$ to~$\rmO(N^\epsilon)$.

The main take away from our implementation of Algorithm~\ref{alg:main_algorithm} and our heuristic analysis of its time complexity in Section~\ref{ssec:time_complexity} is that is the computation of homomorphisms dramatically dominates the runtime. We therefore suggest to develop algorithms that skip this step altogether or at least simplify it by introducing additional structure. In this light, it becomes particularly interesting to compare with the two already existing algorithms by Cohen and Williams. In Section~\ref{sec:alternative_algorithms}, we discuss their assumptions on the representation that vector-valued modular forms are associated with and how they leverage them to gain performance. We also include suggestions on how to relax their assumptions while maintaining the performance advantage. In future work, we hope to come back to this topic.

\section{Preliminaries}%
\label{sec:preliminaries}

We refer to Miyake's book~\cite{miyake-1989} and to the book of Diamond--Shurman~\cite{diamond-shurman-2005} for the basics of classical modular forms. In this section, we revisit vector-valued modular forms and recast some facts about classical ones in a representation theoretic light. We also include a summary of statements from representation theory that we will use. A primer on representation theory of finite groups can be found in Serre's book on representations of finite groups~\cite{serre-1977}, or alternatively Bump's book on Lie groups~\cite{bump-2013}.

We assume standard notation for modular groups, modular forms, and spaces of modular forms.

\subsection{Representation theory}

In this paper we work with finite dimensional, complex left-representations. Let~$G$ be a discrete group. We throughout identify complex representations of~$G$ with the corresponding left-module for the group algebra~$\CC[G]$. Given a representation~$\rho$, we write~$V(\rho)$ for its representation space and~$\rho^\vee$ for its dual. If~$\rho$ and~$\sigma$ are representations of the same group~$G$ that is clear by the context, we allow ourselves to suppress it from the Hom-space notation:
\begin{gather*}
  \Hom(\rho, \sigma)
=
  \Hom{}_G(\rho, \sigma)
\subseteq
  \Hom(V(\rho), V(\sigma))
\tx{.}
\end{gather*}

Given a finite dimensional representation~$\rho$ for a group~$G$, its~$0$\thdash\ cohomology is the space of invariants. We have
\begin{gather}
\label{eq:H0_are_invariants}
  \rmH^0(\rho)
=
  \big\{
  v \in V(\rho) \,:\,
  \forall g \in G .
  \rho(g) v = v
  \big\}
\tx{.}
\end{gather}

Given a finite index subgroup $H \subseteq G$ and an $H$\nbd representations~$\rho$, the induction of~$\rho$ to~$G$ is defined as
\begin{gather*}
  \Ind_H^G\, \rho
=
  \CC[G] \otimes_{\CC[H]} \rho
\tx{.}
\end{gather*}
The restriction of a~$G$\nbd representation~$\rho$ to~$H$ will be denoted by $\Res_H\, \rho = \Res_H^G\, \rho$. As functors, induction and restriction are adjoint via the Frobenius reciprocity (see p.~278 of~\cite{bump-2013}). More precisely, given an~$H$\nbd representation~$\rho$ and a $G$-representation~$\sigma$, we have
\begin{gather}
\label{eq:frobenius_reciprocity}
  \Hom{}_G\big( \Ind_H^G\,\rho,\, \sigma \big)
  \cong
  \Hom{}_H\big( \rho,\, \Res_H^G\,\sigma \big)
\tx{.}
\end{gather}

Mackey's Double Coset Theorem provides a decomposition of restriction--induction in terms of smaller induced representations (see p.~280~of~\cite{bump-2013}). Given two subgroups~$H,K \subseteq G$ and an~$H$\nbd representation~$\rho$, it asserts the isomorphism
\begin{gather}
\label{eq:mackey_double_coset}
  \Res_K^G\, \Ind_H^G\, \rho
\cong
  \bigoplus_{g \in H \backslash G \slash K}
  \Ind_{K \cap H^g}^K\, \rho^g
\tx{,}
\end{gather}
where~$H^g = g^{-1} H g$ is the conjugate of~$H$ and~$\rho^g$ is the pullback of~$\rho$ along the associated conjugation map from~$H^g$ to~$H$, whose action is given by~$\rho^g(h) = \rho(g h g^{-1})$.

We will need several specific representations of~$\SL{2}(\ZZ)$ and its subgroups. For positive integers~$N$, we have the congruence subgroups
\begin{align*}
  \Ga_0(N)
\;&=\;
  \big\{
  \begin{psmatrix} a & b \\ c & d \end{psmatrix} \in \SL{2}(\ZZ) \,:\,
  c \equiv 0 \,\pmod{N}
  \big\}
\tx{,}
\\
  \Ga_1(N)
\;&=\;
  \big\{
  \begin{psmatrix} a & b \\ c & d \end{psmatrix} \in \SL{2}(\ZZ) \,:\,
  c \equiv 0 \,\pmod{N},\,
  a,d \equiv 1 \,\pmod{N}
  \big\}
\tx{,}
\\
  \Ga(N)
\;&=\;
  \big\{
  \begin{psmatrix} a & b \\ c & d \end{psmatrix} \in \SL{2}(\ZZ) \,:\,
  b, c \equiv 0 \,\pmod{N},\,
  a,d \equiv 1 \,\pmod{N}
  \big\}
\tx{,}
\end{align*}
and we have the parabolic subgroup
\begin{gather*}
  \Ga_\infty
\;=\;
  \big\{
  \begin{psmatrix} a & b \\ 0 & d \end{psmatrix} \in \SL{2}(\ZZ)
  \big\}
\tx{.}
\end{gather*}

The one-dimensional trivial representation of~$\Ga \subseteq \SL{2}(\ZZ)$ will be written as~$\bbone$, suppressing the dependence on~$\Ga$ from our notation. A Dirichlet character~$\chi$ modulo~$N$ yields a one-dimensional representation of~$\Ga_0(N)$ via
\begin{gather}
\label{eq:def:dirichlet_type}
  \begin{psmatrix} a & b \\ c & d \end{psmatrix}
\lmto
  \chi(d)
\in
  \GL{1}(\CC)
\tx{.}
\end{gather}
We allow ourselves to equally denote this representation by~$\chi$. Given a positive integer~$N$ and a Dirichlet character~$\chi$ modulo~$N$, we define
\begin{gather}
\label{eq:def:rhoN_rhochi}
  \rho_N
\;=\;
  \Ind_{\Ga_1(N)}^{\SL{2}(\ZZ)}\, \bbone
\quad\tx{and}\quad
  \rho_\chi
\;=\;
  \Ind_{\Ga_0(N)}^{\SL{2}(\ZZ)}\, \chi
\tx{.}
\end{gather}
The kernel of both~$\rho_N$ and~$\rho_\chi$ is the normal core~$\Ga(N)$ of~$\Ga_0(N)$ and~$\Ga_1(N)$ in~$\SL{2}(\ZZ)$. We also record that we have~$\rho_N^\vee \cong \rho_N$ and $\rho_\chi^\vee \cong \rho_{\overline{\chi}}$.

\subsection{Vector-valued modular forms}

We write~$S = \begin{psmatrix} 0 & -1 \\ 1 & 0 \end{psmatrix}$ and~$T = \begin{psmatrix} 1 & 1 \\ 0 & 1 \end{psmatrix}$. We let~$\Ga$ be a finite index subgroup of $\SL{2}(\ZZ)$. An arithmetic type~$\rho$ for~$\Ga$ is a finite-dimensional, complex representation of~$\Ga$. We say that~$\rho$ is a congruence type if its kernel is a congruence subgroup. We refer to the smallest possible possible level of its kernel as the level of~$\rho$. Note that congruence types factor through the quotient group~$\Ga \slash \Ga(N)$ and thus are effectively representations of finite groups. We write~$V(\rho)$ for the representation space of~$\rho$.

We write~$\HS = \{ \tau \in \CC \,:\, \Im(\tau) > 0 \}$ for the Poincaré upper half plane. Let~$k$ be an integer and~$\rho$ an arithmetic type for~$\Ga$. We say that a function $f : \HS \ra V(\rho)$ has moderate growth if there exists an~$a \in \RR$ and a norm~$\|\,\cdot\,\|$ on~$V(\rho)$ such that for all $\ga \in \SL{2}(\ZZ)$ we have uniformly in~$\Re(\tau)$ that
\begin{gather*}
  \big\| f(\ga \tau) \big\|
\ll
  \rmO\big( \Im(\tau)^a \big)
\quad
  \tx{as\ }
  \Im(\tau) \ra \infty
\tx{.}
\end{gather*}

We define a vector-valued slash action for functions~$f : \HS \ra V(\rho)$ by
\begin{gather*}
  \big( f \big|_{k,\rho}\, \ga \big)(\tau)
=
  (c\tau+d)^{-k}\,
  f\big( \mfrac{a\tau + b}{c\tau + d} \big)
\tx{,}\quad
\tx{where\ }
  \ga
=
  \begin{psmatrix} a & b \\ c & d \end{psmatrix}
\in
  \Ga
\tx{.}
\end{gather*}
Now a vector-valued modular form of weight~$k$ and type~$\rho$ is a holomorphic function~$f : \HS \ra V(\rho)$ of moderate growth that satisfies
\begin{gather*}
  \forall \ga \in \Ga \,:\,
  f \big|_{k,\rho}\, \ga = f
\tx{.}
\end{gather*}
The vector-space of such  forms of weight $k$ and type $\rho$ will be written as~$\rmM_k(\rho)$. Note that for every~$k \in \ZZ$, there is some~$\rho$ such that~$\rmM_k(\rho) \ne \{0\}$, for instance one of the symmetric power representations~$\sym^{|k|+3}$ or~$\sym^{|k|+4}$ of~$\SL{2}(\ZZ)$.

The Fourier coefficients of a modular form~$f$ will be denoted by~$c(f;\,n) \in V(\rho)$, $n \in \QQ$. We record a Sturm bound for vector-valued modular forms.
\begin{proposition}[{see~\cite{sturm-1987,bruinier-raum-2015}}]
\label{prop:sturm_bound}
Let~$k$ and~$N$ be positive integers, and~$\rho$ a congruence type for~$\SL{2}(\ZZ)$. Further, let $f \in \rmM_k(\rho)$ be a modular form of weight~$k$ and type~$\rho$. Then we have~$f = 0$ if
\begin{gather*}
  \forall n \in \QQ,\, 0 \le n \leq k \slash 12 \,:\,
  c(f;\, n) = 0
\tx{.}
\end{gather*}
\end{proposition}

\subsection{Eisenstein series}%
\label{ssec:prelim:eisenstein}

In this section, we recall the definition of Eisenstein series for~$\Ga(N)$ and show how they are connected to the arithmetic type~$\rho_N$. We write~$\rmE_k(\Ga)$ for the space of Eisenstein series for a finite index subgroup~$\Ga \subseteq \SL{2}(\ZZ)$.

Given integers~$k > 3$, $N \ge 1$, and~$c, d$ we have a lattice Eisenstein series
\begin{gather}
\label{eq:lattice_eis}
  G_{k,N,c,d}(\tau)
=
  \sum_{\substack{(c',d') \in \ZZ^2 \setminus \{(0,0)\}\\
                  (c',d') \equiv (c,d) \pmod{N}}}
  (c'\tau + d')^{-k}
\in
  \rmE_k(\Ga(N))
\tx{.}
\end{gather}
Since~\eqref{eq:lattice_eis} only depends on~$c$ and~$d$ modulo~$N$, we identify them with their congruence class. In the cases of the weights $1$ and $2$, one obtains a similar Eisenstein series by analytic continuation (see Chapter~7 of~\cite{miyake-1989}). While the weight~$1$ Eisenstein series are holomorphic, the weight~$2$ Eisenstein in general are almost holomorphic. We reserve the notation~$G_{2,N,c,d}$ for the unique holomorphic linear combination~$G^\ahol_{2,N,c,d} + \nu_N G_{2,1,0,0}$, $\nu_N \in \CC$, where~$G^\ahol_{k,N,c,d}$ is the almost holomorphic Eisenstein series. The Fourier expansion of~\eqref{eq:lattice_eis} and its extension to weight~$1$ and~$2$ is known~\cite{diamond-shurman-2005,miyake-1989}. 

An alternative definition of Eisenstein series uses an average over cosets. For~$k > 2$ and~$N \ge 1$, we define
\begin{gather}
  E_{k,N}(\tau)
=
  \sum_{\ga \in \Ga_\infty \backslash \Ga_1(N)}
  1 \big|_k\, \ga
\tx{.}
\end{gather}
As in the case of~$G_{k,N,c,d}$, one extends this via analytic continuation to weight~$1$ and~$2$. We again reserve the notation~$E_{2,N}$ for the holomorphic linear combination~$E^\ahol_{2,N} + \nu_N E_{2,1}$, $\nu_N \in \CC$, where~$E^\ahol_{2,N}$ is the almost holomorphic Eisenstein series.

Given a positive integer~$k$, we let~$\cE_k(N) = \rmE_k(\Ga(N))$ be the space spanned by~$E_{k,N} |_k \ga$ as~$\ga$ runs through~$\SL{2}(\ZZ)$, and set~$\cE_0(N) = \CC$. Now if~$k > 0$, then~$G_{k,N,c,d}$ and~$E_{k,N} |_k \ga$ are nonzero scalar multiplies of one another if~$\ga = \begin{psmatrix} a & b \\ c & d \end{psmatrix} \in \SL{2}(\ZZ)$. In particular, we have
\begin{gather}
\label{eq:eisenspan}
  \cE_k(N)
=
  \linspan \CC \big\{
  G_{k,N,c,d} \,:\,
  0 \le c,d < N, \gcd(c,d,N)=1
  \big\}
\tx{.}
\end{gather}

Per its definition,~$\cE_k(N)$ naturally carries the structure of an~$\SL{2}(\ZZ)$\nbd representation, on which~$\ga$ acts from the left on~$f$ as~$f |_k \ga^{-1}$. The next proposition states the resulting connection between~$\cE_k(N)$ and the arithmetic type~$\rho_N$ defined in~\eqref{eq:def:rhoN_rhochi}. To state it explicitly, we fix the natural basis~$\frake_\ga = \ga \otimes 1$, $\ga \in \SL{2}(\ZZ) \slash \Ga_1(N)$, of~$V(\rho_N)$, where we allow ourselves to identify cosets with their representatives. It gives rise to a dual basis~$\frake_\ga^\vee$ of~$V(\rho_N^\vee)$.

\begin{proposition}
\label{prop:abstract_eis}
Let~$k$ and $N$ be positive integers. Then the following defines a surjective homomorphism of representations:
\begin{gather}
\label{eq:prop:abstract_eis}
  \rho_N^\vee
\lthra
  \cE_k(N)
\tx{,}\quad
  \frake_\ga^\vee
\lmto
  G_{k,N,c,d}
\quad\tx{with\ }
  \ga^{-1} = \begin{psmatrix} a & b \\ c & d \end{psmatrix}
\tx{.}
\end{gather}
\end{proposition}
\begin{proof}
We first check that~\eqref{eq:prop:abstract_eis} is well-defined. The left column of a left-coset representative~$\ga \in \SL{2}(\ZZ) \slash \Ga_1(N)$ is unique modulo~$N$. We conclude that the bottom row of~$\ga^{-1}$ is unique modulo~$N$, and hence~$G_{k,N,c,d}$ in~\eqref{eq:prop:abstract_eis} only depends on the coset of~$\ga$.

As a linear map the given homomorphism is surjective because of~\eqref{eq:eisenspan}. It remains to verify that it is a homomorphism of representations. We let~$\ga, \delta \in \SL{2}(\ZZ)$. Note that we have by definition of the dual
\begin{gather*}
  \rho_N^\vee(\delta)\, \frake_\ga^\vee
=
  \frake_\ga^\vee \circ \rho_N(\delta^{-1})
=
  \frake_{\delta \ga}^\vee
\tx{.}
\end{gather*}
As for the Eisenstein series, we have (see p.~111 of~\cite{diamond-shurman-2005})
\begin{gather*}
  G_{k,N,c,d} \big|_k \delta^{-1}
=
  G_{k,N,c',d'}
\quad\tx{with\ }
  \begin{psmatrix} a' & b' \\ c' & d' \end{psmatrix}
=
  \begin{psmatrix} a & b \\ c & d \end{psmatrix} \delta^{-1}
\tx{.}
\end{gather*}

We combine these equalities to find that
\begin{gather*}
  \rho_N^\vee(\delta)\, \frake_\ga^\vee
=
  \frake_{\delta \ga}^\vee
\lmto
  G_{k,N,c',d'}
=
  G_{k,N,c,d} \big|_k \delta^{-1}
=
  \delta\, G_{k,N,c,d}
\tx{,}
\end{gather*}
where the right hand side features the action of~$\delta$ on the left-representation~$\cE_k(N)$, and
\begin{gather*}
  \ga^{-1}
=
  \begin{psmatrix} a & b \\ c & d \end{psmatrix} 
\tx{,}\quad
  (\delta \ga)^{-1}
=
  \ga^{-1} \delta^{-1}
=
  \begin{psmatrix} a' & b' \\ c' & d' \end{psmatrix}
\tx{.}
\end{gather*}
This confirms that the map~\eqref{eq:prop:abstract_eis} intertwines the action of~$\delta$ and thus finishes the proof.
\end{proof}

\subsection{Products of Eisenstein series}

We will need the next statement of~\cite{raum-xia-2020}. It features tensor products of an arithmetic type~$\rho$ with spaces of Eisenstein series~$\cE_k(N)$. Recall from Section~\ref{ssec:prelim:eisenstein} that we view them as representations for~$\SL{2}(\ZZ)$. Also recall from~\eqref{eq:H0_are_invariants} that the~$0$\thdash\ cohomology describes invariant vectors.

The next lemma allows us to view specific invariants as modular forms.
\begin{lemma}
\label{la:invariants_are_modular_forms}
Consider an arithmetic type~$\rho$ and a subspace~$W \subseteq \rmM_k(\Ga)$ for some weight~$k$ and finite index subgroup~$\Ga \subseteq \SL{2}(\ZZ)$. We assume that~$W$ is stable under the action of\/~$\SL{2}(\ZZ)$ by the weight\nbd$k$ slash action and view it as a left representation via~$(\ga,f) \mto f |_k \ga^{-1}$. Then we have the map
\begin{gather}
\label{eq:la:invariants_are_modular_forms}
  \rmH^0\big( W \otimes \rho \big)\;
\;\lra\;
  \rmM_k(\rho)
\tx{,}\quad
  \sum_i f_i \otimes v_i
\lmto
  \big( \tau \mto \sum_i f_i(\tau) v_i \big)
\tx{.}
\end{gather}
\end{lemma}
\begin{proof}
We let~$\psi$ denote the map in~\eqref{eq:la:invariants_are_modular_forms}. Consider an element~$\sum_i f_i \otimes v_i$, $f_i \in W$, $v_i \in V(\rho)$, of the left hand side of~\eqref{eq:la:invariants_are_modular_forms} and its image~$f = \sum_i f_i v_i$ under~$\psi$. Then the weight\nbd$k$ and type\nbd$\rho$ slash action of~$\ga \in \SL{2}(\ZZ)$ on~$f$ yields
\begin{align*}
  \big( f \big|_{k,\rho}\, \ga \big)(\tau)
&{}=
  \sum_i
  \big( f_i \big|_k \,\ga \big)(\tau)\,
  \rho\big( \ga^{-1} \big) v_i
=
  \sum_i 
  \psi\Big(
   f_i \big|_k \,\ga \,\otimes\, \rho\big( \ga^{-1} \big) v_i
  \Big)(\tau)
\\
&{}=
  \psi\Big(
  \ga^{-1}\,
  \sum_i f_i \otimes v_i
  \Big)(\tau)
=
  \psi \big(\sum_i f_i \otimes v_i\big)(\tau)
=
  f(\tau)
\tx{.}
\end{align*}
\end{proof}

With Lemma~\ref{la:invariants_are_modular_forms} in mind, the next theorem, due to Xià and the second author~\cite{raum-xia-2020}, allows us to identify the invariants on the right hand side of~\eqref{eq:raum_xia} with modular forms.
\begin{theorem}
\label{thm:raum_xia}
Let~$k$, $l$ be integers with~$k \ge 2$ and~$1 \le l \le k - 1$. Fix a congruence type~$\rho$ of level~$N$. Then there is a positive integer~$N_0$ with~$N \isdiv N_0$ such that under the map in Lemma~\ref{la:invariants_are_modular_forms} we have
\begin{gather}
\label{eq:raum_xia}
  \rmM_k(\rho)
\cong
  \rmH^0\big( \cE_k(N) \otimes \rho \big)
  +
  \rmH^0\big( (\cE_l(N_0) \cdot \cE_{k-l}(N_0) ) \otimes \rho \big)
\tx{.}
\end{gather}
\end{theorem}

\subsection*{Acknowledgment}

The authors thank the referee for their remarks that helped to better highlight some aspects of this work.

\section{Main algorithm}
\label{sec:main_algo}

In this section we combine Theorem~\ref{thm:raum_xia} with the surjections onto~$\mathcal{E}_k(N)$ in~\eqref{eq:prop:abstract_eis}, the Fourier expansion of modular forms, and their Sturm bounds to outline our main algorithm.

The main point in employing~\eqref{eq:prop:abstract_eis} is to avoid the computation with functions in favor of symbolic calculations. In particular, we can determine invariant spaces of products of Eisenstein series from invariant spaces of tensor products of induced types. Theorem~\ref{thm:raum_xia} then allows us to determine~$\mathrm{M}_k(\rho)$ from this. We make this connection and intermediate step clear in Theorem~\ref{thm:main_algorithm_preparation}.

We represent modular forms and hence Eisenstein series in terms of their Fourier expansions. This allows us to profit from available, highly optimized implementations of products of power series. It does however require us to set a precision. We use the Sturm bounds for~$\mathrm{M}_k(\rho)$ for this. The details are described in Theorem~\ref{thm:main_algorithm} and Algorithm~\ref{alg:main_algorithm}.

The statement of Theorem~\ref{thm:main_algorithm_preparation} requires some preparation. We have a linear map~$\Phi_\mathcal{E}$ that arises from~\eqref{eq:prop:abstract_eis} by applying it componentwise. More precisely, it is defined by
\begin{gather}
\label{eq:def:thm:main_algorithm_preparation:phiE}
\begin{aligned}
  \Phi_\cE :\,
  \rmH^0 \big( \rho_N^\vee \otimes \rho \big)
  \oplus
  \rmH^0\big( \rho_{N_0}^\vee \otimes \rho_{N_0}^\vee \otimes \rho \big)
&\longrightarrow
  \rmH^0 \big( \cE_k(N) \otimes \rho \big)
  \oplus
  \rmH^0\big( \cE_l(N_0) \otimes \cE_{k-l}(N_0) \otimes\rho \big)
\text{,}\\
  \big(
  \mathfrak{e}_{\ga_1}^\vee \otimes v,\;
  \mathfrak{e}_{\ga_2}^\vee \otimes \mathfrak{e}_{\ga_3}^\vee
    \otimes w
  \big)
&\longmapsto
  \big(
  G_{k,N,c_1,d_1} \otimes v,\;
  G_{l,N_0,c_2,d_2} \otimes G_{k-l,N_0,c_3,d_3} \otimes w
  \big)
\tx{,}
\end{aligned}
\end{gather}
where $\ga_i^{-1}=\begin{psmatrix}a_i&b_i\\c_i&d_i\end{psmatrix}$ for $1\leq i\leq 3$.

Next, the product of modular forms yields a map
\begin{gather}
\label{eq:main_algorithm_preparation:phitimes_product_map}
  \cE_l(N_0) \otimes \cE_{k-l}(N_0)
\lra
  \cE_l(N_0) \cdot \cE_{k-l}(N_0)
\;\subseteq\;
  \rmM_k(\Ga(N_0))
\tx{,}\;
  f_1 \otimes f_2
\lmto
  f_1 \cdot f_2
\tx{.}
\end{gather}
Since it is a homomorphism of~$\SL{2}(\ZZ)$-representations, we obtain a corresponding linear map~$\Phi_\times$ defined by
\begin{gather}
\label{eq:def:thm:main_algorithm_preparation:phitimes}
\begin{aligned}
  \Phi_\times :\,
  \rmH^0\big( \cE_k(N) \otimes \rho \big)
  \oplus
  \rmH^0\big( \cE_l(N_0) \otimes \cE_{k-l}(N_0) \otimes \rho \big)
&\longrightarrow
  \rmH^0\big( \cE_k(N)\otimes\rho \big)
  \oplus
  \rmH^0\big( (\cE_l(N_0) \cdot \cE_{k-l}(N_0)) \otimes \rho \big)
\tx{,}\\
  \big(
  f_1 \otimes v,\;
  f_2 \otimes f_3 \otimes w
  \big)
&\lmto
  \big(
  f_1 \otimes v,\;
  f_2 \cdot f_3 \otimes w
  \big)
\tx{.}
\end{aligned}
\end{gather}

Lemma~\ref{la:invariants_are_modular_forms} allows us to map both direct summands of the right hand side of~\eqref{eq:def:thm:main_algorithm_preparation:phitimes} to~$\rmM_k(\rho)$. In particular, we obtain a map~$\Phi_\Sigma$ by adding up their images.
\begin{gather}
\label{eq:def:thm:main_algorithm_preparation:phisum}
\begin{aligned}
  \Phi_\Sigma :\,
  \rmH^0\big( \mathcal{E}_k(N)\otimes\rho \big)
  \oplus
  \rmH^0\big( (\mathcal{E}_l(N_0)\cdot\mathcal{E}_{k-l}(N_0)) \otimes \rho \big)
&\longrightarrow
  \rmM_k(\rho)
\tx{,}\\
  \big(
  f_1 \otimes v,\;
  f_2 \cdot f_3 \otimes w
  \big)
&\lmto
  f_1 \cdot v + f_2 \cdot f_3 \cdot w
\tx{.}
\end{aligned}
\end{gather}

\begin{theorem}%
\label{thm:main_algorithm_preparation}
Let~$k$, $l$ be integers with~$k \ge 2$ and~$1 \le l \le k - 1$, and fix a congruence type~$\rho$ of level~$N$, and a positive integer~$N_0$. Then the linear maps~\eqref{eq:def:thm:main_algorithm_preparation:phiE}, \eqref{eq:def:thm:main_algorithm_preparation:phitimes}, and~\eqref{eq:def:thm:main_algorithm_preparation:phisum} are surjective. In particular, if~$N_0$ is chosen as in Theorem~\ref{thm:raum_xia}, by composing additionally with~\eqref{eq:la:invariants_are_modular_forms}, we have a surjective map
\begin{gather}
\label{eq:thm:main_algorithm_preparation}
\begin{aligned}
  \rmH^0\big( \rho_N^\vee \otimes \rho \big)
  \,\oplus\,
  \rmH^0\big( \rho_{N_0}^\vee \otimes \rho_{N_0}^\vee \otimes \rho \big)
&
\xlongtwoheadrightarrow{\Phi_{\mathcal{E}}}
  \rmH^0\big( \cE_k(N) \otimes \rho \big)
  \,\oplus\,
  \rmH^0\big( \cE_l(N_0) \otimes \cE_{k-l}(N_0) \otimes \rho \big)
\\
&
\xlongtwoheadrightarrow{\Phi_\times}
  \rmH^0\big( \cE_k(N) \otimes \rho \big)
  \,\oplus\,
  \rmH^0\big( (\cE_l(N_0) \cdot \cE_{k-l}(N_0)) \otimes \rho \big)
\\
&
\xlongtwoheadrightarrow{\Phi_\Sigma}
  \rmH^0\big( \cE_k(N) \otimes \rho \big)
  \,+\,
  \rmH^0\big( (\cE_l(N_0) \cdot \cE_{k-l}(N_0)) \otimes \rho \big)
\lthra
  \rmM_k(\rho)
\tx{.}
\end{aligned}
\end{gather}
\end{theorem}
\begin{proof}
We argue one by one that each map in the composition is surjective. Recall that given a homomorphism of finite dimensional and semi-simple representations~$\rho \ra \sigma$, the corresponding map on invariants~$\rmH^0(\rho) \ra \rmH^0(\sigma)$ is surjective, if~$\rho \ra \sigma$ is surjective. Since both~$\rho_M^\vee$ and~$\cE_k(M)$ for any positive integer~$M$ are finite dimensional representations of~$\SL{2}(\ZZ)$ that factor through the finite quotient group~$\SL{2}(\ZZ) \slash \Ga(M)$, they are semi-simple. We combine this with the homomorphism in~\eqref{eq:prop:abstract_eis} to conclude that
\begin{gather*}
  \rmH^0 \big( \rho_N^\vee \otimes \rho \big)
\lra
  \rmH^0 \big( \cE_k(N) \otimes \rho \big)
\quad\tx{and}\quad
  \rmH^0\big( \rho_{N_0}^\vee \otimes \rho_{N_0}^\vee \otimes \rho \big)
\lra
  \rmH^0\big( \cE_l(N_0) \otimes \cE_{k-l}(N_0) \otimes\rho \big)
\end{gather*}
are surjective, and therefore the map~$\Phi_\cE$ in~\eqref{eq:def:thm:main_algorithm_preparation:phiE} is surjective.

The product map~\eqref{eq:main_algorithm_preparation:phitimes_product_map} is surjective by definition of its codomain. Since the codomain is contained in the finite dimensional space~$\rmM_k(\Ga(N_0))$ and it factors as a representation through~$\SL{2}(\ZZ) \slash \Ga(N_0)$, it is semi-simple. We conclude that the map~$\Phi_\times$ in~\eqref{eq:def:thm:main_algorithm_preparation:phitimes} is surjective. In particular, the map~$\Phi_\Sigma$ in~\eqref{eq:def:thm:main_algorithm_preparation:phisum} is surjective by the definition of its codomain. The final map in~\eqref{eq:thm:main_algorithm_preparation} arises from~\eqref{eq:la:invariants_are_modular_forms} in Lemma~\ref{la:invariants_are_modular_forms}. It is surjective by Theorem~\ref{thm:raum_xia}.
\end{proof}

\subsection{Fourier expansions}
\label{ssec:main_algorithm:fourier_expansions}

To derive Algorithm~\ref{alg:main_algorithm} from Theorem~\ref{thm:main_algorithm_preparation}, we employ truncated Fourier expansions with coefficients in~$\QQab$. Given any ring~$R$ and rational number~$B \in \QQ$, we write
\begin{gather*}
  \rmFE(R)
\;=\;
  \bigcup_{N \in \ZZ_{>0}}
  R\big\llbrkt q^{\frac{1}{N}} \big\rrbrkt \big[q^{-1}\big]
\quad\tx{and}\quad
  \rmFE_B(R)
\;=\;
  \rmFE(R)
  \big\slash
  q^B
  \bigcup_{N \in \ZZ_{>0}}
  R \big\llbrkt q^{\frac{1}{N}} \big\rrbrkt
\end{gather*}
for the ring of Puiseux series with coefficients in~$R$ and its quotient by series of valuation at least~$B$. If~$\rho(T)$ is diagonalizable the Fourier expansion of modular forms yields a map
\begin{gather}
\label{eq:def:fourier_expansion_map}
  \rmfe :\,
  \rmM_k(\rho)
\lra
  \rmFE(\CC) \otimes V(\rho)
\tx{,}\quad
  f
\lmto
  \rmfe(f)
:=
  \sum_{n \in \QQ} c(f;\,n) q^n
\tx{.}
\end{gather}
If~$\rho$ has level~$N$, then exponents of~$q$ in the support of the image have denominator at most~$N$. For fixed~$k$ and~$\rho$, the Sturm bound for vector-valued modular forms (see Proposition~\ref{prop:sturm_bound}) yields some explicit nonnegative~$P \in \QQ$ such that the resulting ``truncated'' Fourier expansion map is injective:
\begin{gather}
\label{eq:def:truncated_fourier_expansion_map}
  \rmfe_P :\,
  \rmM_k(\rho)
\lra
  \rmFE_P(\CC) \otimes V(\rho)
\tx{,}\quad
  f
\lmto
  \rmfe_P(f)
:=
  \sum_{\substack{n \in \QQ\\n < P}} c(f;\,n) q^n
\tx{.}
\end{gather}

For given~$P \in \QQ$, we can apply~$\rmfe_P$ to each space in the middle and right column of~\eqref{eq:thm:main_algorithm_preparation}. Since Fourier expansions are ring homomorphisms, that is, they are compatible with multiplication and addition, the following diagram commutes, where the first bottom arrow is the multiplication of Puiseux series:
\begin{center}
\begin{tikzpicture}
\matrix(m)[matrix of math nodes,
row sep = 2em, column sep = 3em,
text height = 1.5em, text depth = 1.5ex]
{  \rmH^0\big( \cE_l(N_0) \otimes \cE_{k-l}(N_0) \otimes \rho \big)
&  \rmH^0\big( (\cE_l(N_0) \cdot \cE_{k-l}(N_0)) \otimes \rho \big) 
&  \rmM_k(\rho)
\\ \rmFE_P(\CC) \otimes \rmFE_P(\CC) \otimes V(\rho)
&  \rmFE_P(\CC) \otimes V(\rho)
&  \rmFE_P(\CC) \otimes V(\rho)
\\};

\path[-stealth]
(m-1-1) edge (m-1-2)
(m-1-2) edge (m-1-3)
(m-2-1) edge (m-2-2)
(m-2-2) edge (m-2-3)
(m-1-1) edge node[right] {$\rmfe_P \otimes \rmfe_P \otimes \id$}(m-2-1)
(m-1-2) edge node[right] {$\rmfe_P \otimes \id$}(m-2-2)
(m-1-3) edge node[right] {$\rmfe_P$} (m-2-3);
\end{tikzpicture}
\end{center}
A similar diagram holds for the other terms that occur in~\eqref{eq:thm:main_algorithm_preparation}.

For simplicity and by slight abuse of notation, we write~$\rmfe_P \circ \Phi_\cE$ for the Fourier expansion map $( \rmfe_P \otimes \id_\rho \,\oplus\, \rmfe_P \otimes \rmfe_P \otimes \id_\rho ) \circ \Phi_\cE$. Further, we write~$\Phi_\times$ and~$\Phi_\Sigma$ for the multiplication and addition maps on spaces of Puiseux series that correspond to the maps in~\eqref{eq:thm:main_algorithm_preparation}. Then we have the equality
\begin{gather}
\label{eq:thm:main_algorithm_preparation:intertwine_fourier_expansion}
  \rmfe_P \circ \Phi_\Sigma \circ \Phi_\times \circ \Phi_\cE
\;=\;
  \Phi_\Sigma \circ \Phi_\times \circ \rmfe_P \circ \Phi_\cE
\tx{.}
\end{gather}
In other words, we can intertwine the Fourier expansion with the maps in Theorem~\ref{thm:main_algorithm_preparation}.

\subsection{Algebraic Fourier coefficients}
\label{ssec:main_algorithm:algebraic_fourier_coefficients}

Assume that we have a~$\QQab$-structure~$V(\rho, \QQab) \subset V(\rho)$, that is, $V(\rho) = V(\rho, \QQab) \otimes_{\QQab} \CC$ and~$V(\rho, \QQab)$ is stable under the action of~$\rho$. Then we define
\begin{gather*}
  \rmM_k\big(\rho, \QQab\big)
:=
  \big\{
  f \in \rmM_k(\rho) \,:\,
  \forall n \in \QQ \,.\, c(f;\,n) \in V\big( \rho,\QQab \big)
  \big\}
\quad\tx{and}\quad
  \cE_k(N, \QQab)
=
  \cE_k(N) \cap \rmM_k(\Ga(N), \QQab)
\tx{,}
\end{gather*}
where we suppress the dependence on~$V(\rho,\QQab)$ from our notation. If~$\rho$ is a congruence type, then by for example Deligne-Rapoport~\cite{deligne-rapoport-1973} we have
\begin{gather}
  \rmM_k(\rho)
=
  \rmM_k(\rho, \QQab) \otimes_{\QQab} \CC
\tx{.}
\end{gather}
The representations~$\rho_N^\vee$ have~$\QQab$\nbd structures via their identification with permutation representations, which are compatible with the map~\eqref{eq:prop:abstract_eis} to~$\cE_k(N)$ and the respective~$\QQab$-structures of their images.

We also have~$\QQab$\nbd structures of the domain of~\eqref{eq:thm:main_algorithm_preparation} in Theorem~\ref{thm:main_algorithm_preparation}. In particular, we obtain
\begin{gather}
\label{eq:thm:main_algorithm_preparation:lhr_rational_structure}
  \rmH^0\big( \rho_N^\vee \otimes \rho \big)
=
  \rmH^0\big( \rho_N^\vee \otimes \rho,\, \QQab \big)
  \otimes_{\QQab} \CC
\quad\tx{and}\quad
  \rmH^0\big( \rho_{N_0}^\vee \otimes \rho_{N_0}^\vee \otimes \rho \big)
=
  \rmH^0\big( \rho_{N_0}^\vee \otimes \rho_{N_0}^\vee \otimes \rho,\, \QQab \big)
  \otimes_{\QQab} \CC
\tx{.}
\end{gather}
We use similar notation for the other spaces in~\eqref{eq:thm:main_algorithm_preparation}.
By definition, the~$\QQab$-structures of modular forms are compatible with products and sums. In particular, the maps~$\Phi_\times$ and~$\Phi_\Sigma$ both descend to maps of the~$\QQab$-structures. Thus Theorem~\ref{thm:main_algorithm_preparation} yields a surjective map
\begin{gather}
\label{eq:thm:main_algorithm_preparation:rational_structure}
  \Phi_\Sigma \circ \Phi_\times \circ \Phi_\cE :\,
  \rmH^0\big( \rho_N^\vee \otimes \rho,\, \QQab \big)
  \,\oplus\,
  \rmH^0\big( \rho_{N_0}^\vee \otimes \rho_{N_0}^\vee \otimes \rho,\, \QQab \big)
\lra
  \rmM_k\big( \rho, \QQab \big)
\tx{.}
\end{gather}

\subsection{Statement of the algorithm}

Combining Theorem~\ref{thm:main_algorithm_preparation} with the discussion in Sections~\ref{ssec:main_algorithm:fourier_expansions} and~\ref{ssec:main_algorithm:algebraic_fourier_coefficients}, we prove correctness of our main algorithm.

\begin{theorem}
\label{thm:main_algorithm}
Fix a weight~$k \ge 2$ and a congruence type~$\rho$. Assume that~$P \in \QQ$ is at least the Sturm bound for~$k$ and~$\rho$. Then Algorithm~\ref{alg:main_algorithm} computes a basis for the image of\/~$\rmM_k(\rho) \ra \rmFE_P(\CC) \otimes V(\rho)$ that is contained in~$\rmFE_P(\QQab) \otimes V(\rho, \QQab)$.
\end{theorem}

\begin{algorithm}[H]
\label{alg:main_algorithm}
let~$N$ the level of~$\rho$,
$N_0$ as in Theorem~\ref{thm:raum_xia}, and
$P \leftarrow \lceil P N \rceil \slash N$\;
let~$v_i$, $1 \le i \le \dim(\rho)$ be a basis of~$V(\rho,\QQab)$\;
let $M$ be the matrix of size~$0 \times P N \dim(\rho)$ over~$\QQab$\;
\label{alg:row:main_algorithm:compute_invariants}
let $\cB = \bigcup_{i=1}^I \cB_i$ be a disjoint decomposition of a basis of
  $\rmH^0(\rho_N^\vee \otimes \rho,\, \QQab)
  \oplus
  \rmH^0(\rho_{N_0}^\vee \otimes \rho_{N_0}^\vee \otimes \rho,\, \QQab)$\;
\For{$1 \le i \le I$}{
  \For{$b \in \cB_i$}{
    let~$f = \sum_i f_i v_i \leftarrow \Phi_\Sigma\, \Phi_\times\, \rmfe_P\, \Phi_\cE(b)$\;
    \label{alg:row:main_algorithm:append_row}
    append to~$M$ the row with entries~$r_{i,n} \leftarrow c(f_i;\, n \slash P)$, $0 \le n < N P$\, $1 \le i \le \dim(\rho)$;
  }
  \label{alg:row:main_algorithm:row_echelon_reduction}
  replace~$M$ by its reduced row echelon form\;
  \If{$\rank\,M = \dim\,\rmM_k(\rho)$
    \label{alg:row:main_algorithm:dimension_comparison}
    }{
    for each row~$r$ of~$M$,
    output the truncated Fourier expansion~$f$ with coefficient~$c(f_i;\,n) = r_{i,n}$\;
    \Return;
  }
}
\caption{Computing a basis for~$\rmM_k(\rho)$}
\end{algorithm}

\begin{remark}
\label{rm:main_algorithm:symbolic_expressions}
By keeping track of the formal linear combinations of elements~$b \in \cB$ that equal the rows of~$M$ in each stage of Algorithm~\ref{alg:main_algorithm}, we can also compute an expression for each element of the basis of~$\rmM_k(\rho)$ in terms of products of Eisenstein series.
\end{remark}

\begin{proof}%
[Proof of Theorem~\ref{thm:main_algorithm}]
To see that Algorithm~\ref{alg:main_algorithm} terminates, it suffices to note that~$\cB$ is finite. We have to show that its output is correct.

We can and will assume that~$P \in \frac{1}{N} \ZZ$. We let~$\lambda$ be the linear map from~$\rmFE_P(\QQab) \otimes V(\rho)$ to~$\QQ^{\mathrm{ab}\,PN}$. We impose an ordering~$b_1, \ldots, b_{J_i}$ on~$\cB_i$. By induction on~$1 \le i \le I$ and~$1 \le j \le J_i$, one shows that after line~\ref{alg:row:main_algorithm:append_row} of Algorithm~\ref{alg:main_algorithm} the row span of~$M$ equals the span of
\begin{gather*}
  \lambda\, \Phi_\Sigma\, \Phi_\times\, \rmfe_P\, \Phi_\cE
  \big( \cB_1 \cup \cdots \cup \cB_{i-1} \cup \{b_1, \ldots, b_j\} \big)
\tx{.}
\end{gather*}

By the relation in~\eqref{eq:thm:main_algorithm_preparation:intertwine_fourier_expansion}, we can intertwine~$\rmfe_P$ with~$\Phi_\Sigma\, \Phi_\times$ in the previous expression. Lemma~\ref{la:invariants_are_modular_forms} implies that~$\Phi_\Sigma\, \Phi_\times\,\Phi_\cE(b)$ lies in~$\rmM_k(\rho)$ for every~$b \in \cB$. The compatibility with~$\QQab$-structures in~\eqref{eq:thm:main_algorithm_preparation:rational_structure} implies that it even lies in~$\rmM_k(\rho, \QQab)$. Since~$\rho$ has level~$N$, the Fourier coefficients of any element of~$\rmM_k(\rho)$ of index~$n \in \QQ$ vanish if~$n \not\in \frac{1}{N}\ZZ$. They also vanish if~$n$ is negative. Therefore, $\lambda\, \rmfe_P$ is an isomorphism from~$\rmM_k(\rho, \QQab)$ onto its image. We conclude that after line~\ref{alg:row:main_algorithm:append_row} of Algorithm~\ref{alg:main_algorithm} the row span of~$M$ is isomorphic to the span of
\begin{gather*}
  \Phi_\Sigma\, \Phi_\times\, \Phi_\cE
  \big( \cB_1 \cup \cdots \cup \cB_{i-1} \cup \{b_1, \ldots, b_j\} \big)
\in
  \rmM_k(\rho)
\tx{.}
\end{gather*}

From what we have shown, we have~$\rank\,M \le \dim\,\rmM_k(\rho)$ at every stage of Algorithm~\ref{alg:main_algorithm}. By~\eqref{eq:thm:main_algorithm_preparation:lhr_rational_structure}, the basis~$\cB$ is a basis for the domain of~\eqref{eq:thm:main_algorithm_preparation}. Now Theorem~\ref{alg:main_algorithm} implies that
\begin{gather*}
  \rmM_k(\rho)
=
  \linspan \CC \big\{
  \Phi_\Sigma\, \Phi_\times\, \Phi_\cE(b) \,:\,
  b \in \cB_i, 1 \le i \le I
  \big\}
\tx{.}
\end{gather*}
In particular, after processing all~$b \in \cB$ in, we have~$\rank\,M = \dim\,\rmM_k(\rho,\QQab)$. If this equality holds in line~\ref{alg:row:main_algorithm:dimension_comparison}, then the row span of~$M$ equals~$\rmM_k(\rho, \QQab)$, and therefore its rows yield a basis as desired.
\end{proof}

\section{Decompositions}
\label{sec:decompositions}

Many of the implementation details for Algorithm~\ref{alg:main_algorithm} depend on suitable decompositions of various arithmetic types and representations that appear. In this section, we discuss the isotypic decompositions that are relevant to the left hand side of~\eqref{eq:thm:main_algorithm_preparation} in Section~\ref{ssec:isotypic_decomposition} and a decomposition into induced representations that holds for specific arithmetic types in Section~\ref{ssec:twisted_permutation_types}. In Section~\ref{ssec:translation_orbits_fourier_expansions}, we use the results from the latter one to analyze the support of the Fourier expansion of vector-valued modular forms.

The results obtained in this section are essential ingredients for our implementation of Algorithm~\ref{alg:main_algorithm}, which we discuss in Section~\ref{sec:implementation}.

\subsection{Isotypic decomposition}
\label{ssec:isotypic_decomposition}

In this section, we describe how to decompose the invariant spaces $\rmH^0(\rho_N^\vee\otimes\rho)$ and~$\rmH^0(\rho_{N_0}^\vee\otimes\rho_{N_0}^\vee\otimes\rho)$ by using standard tools from representation theory. Specifically, we employ Frobenius reciprocity, Mackey's double coset decomposition, and the isotypic decomposition of intermediate arithmetic types for~$\Ga_1(N)$ and~$\Ga_0(N)$. This allows us to characterize~$\rmH^0(\rho_N^\vee\otimes\rho)$ as the~$\bbone$-isotypic component of~$\Res_{\Ga_1(N)}(\rho)$, and~$\rmH^0(\rho_{N_0}^\vee\otimes\rho_{N_0}^\vee\otimes\rho)$ as a direct sum of tensor products, indexed by double cosets in~$\Ga_1(N_0) \backslash \SL{2}(\ZZ) \slash \Ga_1(N_0)$, which we detail in~\eqref{prop:invariants_via_Ga1_restriction:triple_product}.

Recall first that every semi-simple representation can be decomposed into a direct sum of its so-called isotypic components. In particular, for a congruence type $\rho$ of level $N$ for $\Ga_1(N)$, we obtain the following decomposition:
\begin{gather}
\label{eq:Gamma1_isotypic}
  \rho
\cong
  \bigoplus_{n \pmod{N}}
  \rho\isotypGaN{n}
\tx{,}
\end{gather}
since $\Ga_1(N) \slash \Ga(N) = T^\ZZ \Ga(N)$ (see for example~p.~13 of~\cite{diamond-shurman-2005} or p.~8 of~\cite{stein-2007}). Here, the isotypic component~$\rho[e(n \slash N)] \subseteq \rho$ equals the direct sum of all irreducible subrepresentations of~$\rho$ isomorphic to the one-dimensional representation~$\ga \mto e(n b \slash N) \in \GL{1}(\CC)$, $\ga = \begin{psmatrix} a & b \\ c & d \end{psmatrix}$. In other words, we have
\begin{gather*}
  V\big( \rho\isotypGaN{n} \big)
=
  \Hom\big(
  \begin{psmatrix} a & b \\ c & d \end{psmatrix}
  \mto
  e\big( \mfrac{n b}{N} \big),\,
  \rho
  \big) (\CC)
\subseteq
  V(\rho)
\tx{.}
\end{gather*}
We refer to the decomposition~\eqref{eq:Gamma1_isotypic} as the $\Ga_1(N)$-isotypic decomposition of $\rho$.

The decomposition of the invariant spaces~$\rmH^0(\rho_N^\vee\otimes\rho)$ and~$\rmH^0(\rho_{N_0}^\vee\otimes\rho_{N_0}^\vee\otimes\rho)$ is given in the next proposition. In its statement we use the abbreviation
\begin{gather}
\label{eq:def:pi_g_permutation}
  \pi_g
=
  \Ind_{\Ga_1(N_0) \cap g^{-1}\Ga_1(N_0)g}^{\Ga_1(N_0)}\, \bbone
\tx{,}\quad
  g \in \Ga_1(N_0) \backslash \SL{2}(\ZZ) \slash \Ga_1(N_0)
\tx{.}
\end{gather}
Note that this is a permutation representation, that is, we can and will identify its image with a subgroup of permutations of~$\Ga_1(N_0) \slash (\Ga_1(N_0)\cap g^{-1}\Ga_1(N_0)g )$.

\begin{proposition}
\label{prop:invariants_via_Ga1_restriction}
Let~$\rho$ be a congruence type of level~$N$, and~$N_0$ a positive integer as in Theorem~\ref{thm:raum_xia}. Then we have
\begin{gather}
\label{prop:invariants_via_Ga1_restriction:double_product}
  \rmH^0\big( \rho_N^\vee \otimes \rho \big)
\;\cong\;
  \big(\Res_{\Ga_1(N)}\,\rho \big)[\bbone]
\tx{,}
\end{gather}
and
\begin{align}
\label{prop:invariants_via_Ga1_restriction:triple_product}
  \rmH^0\big( \rho_{N_0}^\vee \otimes \rho_{N_0}^\vee \otimes \rho \big)
\;\cong\;
  \bigoplus_{\substack{
     g \in \Ga_1(N_0) \backslash \SL{2}(\ZZ) \slash \Ga_1(N_0)\\
     m_0 \pmod{N_0},\, m \pmod{N}\\
     m_0 \equiv -m N_0 \slash N \pmod{N_0}}}
  \pi_g\isotypGasub{m_0}{N_0}
  \,\otimes\,
  \Res_{\Ga_1(N_0)}\Big(
  \big( \Res_{\Ga_1(N)}\,\rho \big)\isotypGaN{m}
  \Big)
\tx{.}
\end{align}
\end{proposition}
\begin{proof}
We begin by showing~\eqref{prop:invariants_via_Ga1_restriction:double_product}. Note first that~$\rmH^0(\rho_N^\vee\otimes\rho)\cong\mathrm{Hom}(\rho_N,\rho)$. By Frobenius reciprocity in~\eqref{eq:frobenius_reciprocity}, we obtain $\mathrm{Hom}(\rho_N,\rho) \cong \rmH^0(\Res_{\Ga_1(N)}(\rho))$. From the definition of the~$\bbone$-isotypic component, we see that
\begin{gather*}
  \rmH^0\big( \Res_{\Ga_1(N)}\,\rho \big)
\cong
  \big( \Res_{\Ga_1(N)}\,\rho \big)[\bbone]
\tx{.}
\end{gather*}
This establishes~\eqref{prop:invariants_via_Ga1_restriction:double_product}.

As for $\rmH^0(\rho_{N_0}^\vee\otimes\rho_{N_0}^\vee\otimes\rho)$ in~\eqref{prop:invariants_via_Ga1_restriction:triple_product}, we first employ~\eqref{eq:frobenius_reciprocity} to obtain
\begin{gather}
\label{eq:iso_final:frob_applied}
  \rmH^0\big( \rho_{N_0}^\vee\otimes\rho_{N_0}^\vee\otimes\rho \big)
\cong
  \rmH^0\big( 
  \Res_{\Ga_1(N_0)}\, \rho_{N_0}^\vee
  \,\otimes\,
  \Res_{\Ga_1(N_0)}\, \rho
  \big)
\tx{.}
\end{gather}
In addition, Mackey's double coset decomposition in~\eqref{eq:mackey_double_coset} implies that
\begin{gather}
\label{eq:iso_final:mackey_applied}
  \Res_{\Ga_1(N_0)}\, \rho_{N_0}
\cong
  \bigoplus_{g \in \Ga_1(N_0) \backslash \SL{2}(\ZZ) \slash \Ga_1(N_0)}
  \Ind_{\Ga_1(N_0)\cap g^{-1}\Ga_1(N_0)g}^{\Ga_1(N_0)}\, \bbone
\tx{.}
\end{gather}
In the argument on the right hand side, we recognize the permutation representation~$\pi_g$ defined in~\eqref{eq:def:pi_g_permutation}. The self-duality~$\rho_{N_0}^\vee\cong\rho_{N_0}$ allows us to insert~\eqref{eq:iso_final:mackey_applied} into~\eqref{eq:iso_final:frob_applied}. We obtain
\begin{gather}
\label{eq:iso_final:mackey_inserted_into_frob_applied}
  \rmH^0\big(
  \Res_{\Ga_1(N_0)}\, \rho_{N_0}^\vee
  \,\otimes\,
  \Res_{\Ga_1(N_0)}\, \rho
  \big)
\;\cong\;
  \bigoplus_{g \in \Ga_1(N_0) \backslash \SL{2}(\ZZ) \slash \Ga_1(N_0)}
  \rmH^0\big(
  \pi_g \otimes \Res_{\Ga_1(N_0)}\,\rho
  \big)
\tx{.}
\end{gather}

In order to apply the isotypic decomposition for~$\Ga_1(N)$-representations, we rewrite the restriction to~$\Ga_1(N_0)$ on the right hand side of~\eqref{eq:iso_final:mackey_inserted_into_frob_applied} as a restriction in steps to~$\Ga_1(N)$ and then~$\Ga_1(N_0)$. The $\Ga_1(N_0)$-isotypic decomposition of~$\pi_g$ and the~$\Ga_1(N)$-isotypic decomposition of~$\rho$ as in~\eqref{eq:Gamma1_isotypic} yield
\begin{multline}
\label{eq:iso_final:tensor_isotypcial_components}
  \bigoplus_{g \in \Ga_1(N_0) \backslash \SL{2}(\ZZ) \slash \Ga_1(N_0)}
  \rmH^0\Big(
  \pi_g
  \,\otimes\,
  \Res_{\Ga_1(N_0)} \big( \Res_{\Ga_1(N)}\, \rho \big)
  \Big)
\\
\cong\;
  \bigoplus_{\substack{
    g \in \Ga_1(N_0) \backslash \SL{2}(\ZZ) \slash \Ga_1(N_0) \\
    m_0 \pmod{N_0},\, m \pmod{N}}}
  \rmH^0\Big(
  \pi_g\isotypGasub{m_0}{N_0}
  \,\otimes\,
  \Res_{\Ga_1(N_0)} \Big(
  \big( \Res_{\Ga_1(N)}\,\rho \big)\isotypGaN{m}
  \Big)
  \Big)
\tx{.}
\end{multline}

The tensor product in the argument on the right hand side of~\eqref{eq:iso_final:tensor_isotypcial_components} is isomorphic to
\begin{gather*}
  \begin{psmatrix} a & b \\ c & d \end{psmatrix}
\lmto
  e\Big(
  \mfrac{(m_0 + m N_0 \slash N)\, b}{N_0}
  \Big)
\in
  \GL{1}(\CC)
\tx{.}
\end{gather*}
In particular, to isolate the isotrivial component in~\eqref{eq:iso_final:tensor_isotypcial_components}, it suffices to impose the congruence condition $m_0 \equiv -m N_0 \slash N \,\pmod{N_0}$ in the direct sum. In conclusion, we obtain
\begin{align*}
  \rmH^0\big( \rho_{N_0}^\vee \otimes \rho_{N_0}^\vee \otimes \rho \big)
\;\cong\;
  \bigoplus_{\substack{
    g \in \Ga_1(N_0) \backslash \SL{2}(\ZZ) \slash \Ga_1(N_0) \\
    m_0 \pmod{N_0},\, m \pmod{N} \\
    m_0 \equiv -mN_0 \slash N \pmod{N_0}}}
  \pi_g\isotypGasub{m_0}{N_0}
  \,\otimes\,
  \Res_{\Ga_1(N_0)} \Big(
  \big( \Res_{\Ga_1(N)}\,\rho \big)\isotypGaN{m}
  \Big)
\tx{.}
\end{align*}
\end{proof}

To estimate the runtime of Algorithm~\ref{alg:main_algorithm} in Section~\ref{ssec:time_complexity}, we need to bound the dimension of the invariant spaces~$\rmH^0( \rho_N^\vee \otimes \rho)$ and~$\rmH^0(\rho_{N_0}^\vee \otimes \rho_{N_0}^\vee \otimes \rho)$. To this end, we need to know the structure of the permutation~$\pi_g(T)$ defined in~\eqref{eq:def:pi_g_permutation}. The following lemma provides this structure.

\begin{lemma}
\label{la:pi_g_groups}
Given a positive integer~$N$, we have that
\begin{gather}
\label{eq:la:pi_g_groups}
\begin{split}
  \Ga_1(N) \,\big\slash\, (\Ga_1(N) \cap g^{-1} \Ga_1(N) g)
\;\cong\;
  \big( \Ga_1(N) \slash \Ga(N) \big)
  \,\big\slash\,
  \big( (\Ga_1(N) \cap g^{-1} \Ga_1(N) g) \slash \Ga(N) \big)
\tx{,}
\\
  \Ga_1(N) \slash \Ga(N)
\;=\;
  T^\ZZ\, \Ga(N)
\tx{,}\qquad
  (\Ga_1(N) \cap g^{-1}\Ga_1(N)g) \slash \Ga(N)
\;=\;
  T^{n_g \ZZ}\, \Ga(N)
\tx{,}
\end{split}
\end{gather}
with
\begin{gather*}
  n_g
=
  \lcm\big( \mfrac{N}{\gcd(N,ac)},\, \mfrac{N}{\gcd(N,c^2)} \big)
\tx{,}\quad 
  g
=
  \begin{psmatrix} a & b\\ c & d \end{psmatrix}
\in
  \SL{2}(\ZZ)
\tx{.}
\end{gather*}
\end{lemma}
\begin{proof}
The isomorphism is a direct application of Noether's third isomorphism theorem (see for example~\cite{dummit_foote-2004}). The equality~$\Ga_1(N) \slash \Ga(N) = T^\ZZ\, \Ga(N)$ was already explained. Hence, we only need prove the last equality.

To this end, let~$\delta \in \Ga_1(N) \cap g^{-1}\Ga_1(N)g$. Since~$\delta \in \Ga_1(N)$, we have that~$\delta = T^n \delta'$ for some integer~$n$ and some~$\delta' \in \Ga(N)$. We also have that~$g \delta g^{-1} = g T^n \delta' g^{-1} \in \Ga_1(N)$. Since~$\Ga(N) \subseteq \Ga_1(N)$ is normal, this is equivalent to~$g T^n g^{-1} \in \Ga_1(N)$. Inserting the entries of~$g$, we discover that this is further equivalent to the congruences~$a c n  \equiv 0 \,\pmod{N}$ and~$c^2 n \equiv 0 \,\pmod{N}$. This can be rephrased as
\begin{gather*}
  \lcm\big( \mfrac{N}{\gcd(N,ac)},\, \mfrac{N}{\gcd(N,c^2)} \big)
\isdiv
  n
\tx{.}
\end{gather*}
In other words, we find that~$\delta \in T^{n_g \ZZ} \Ga(N)$. We find that~$T^{n_g \ZZ} \Ga(N) \subseteq (\Ga_1(N) \cap g^{-1} \Ga_1(N) g) \slash \Ga(N)$, when reverting the previous calculation, and thus finish the proof.
\end{proof}

Lemma~\ref{la:pi_g_groups} directly implies that~$\pi_g(T)$ corresponds to a transitive, that is, cyclic, permutation of order~$n_g$. Hence $\pi_g(T)$ has distinct eigenvalues~$e(m \slash n_g)$ for~$m \,\pmod{n_g}$. The corresponding isotypic components~$\pi_g[e(m \slash n_g)] \subseteq \rho$ are therefore at most one-dimensional. Given~$v \in V(\pi_g)$, we can project to them, and obtain
\begin{gather}
\label{eq:def:pi_g_groups_eigenvector}
  v\isotypGasub{m}{n_g}
=
  \sum_{h \pmod{n_g}} e\big( \mfrac{-m h}{n_g} \big)\, \pi_g(T)^h v
\in
  V\big( \pi_g\isotypGasub{m}{n_g} \big)
\tx{.}
\end{gather}

We finish this section with a bound on the dimension of~$\rmH^0(\rho_{N_0}^\vee\otimes\rho_{N_0}^\vee\otimes\rho)$ that will important in Section~\ref{ssec:time_complexity}.

\begin{proposition}
\label{prop:triple_product_dim_bound}
We have
\begin{gather*}
  \dim\, \rmH^0\big(
  \rho_{N_0}^\vee \otimes \rho_{N_0}^\vee \otimes \rho
  \big)
\;\le\;
  \#\big( \Ga_1(N_0) \backslash \SL{2}(\ZZ) \slash \Ga_1(N_0) \big)\,
  \dim(\rho)
\;\ll\;
  N_0^{1+\epsilon}\, \dim(\rho)
\tx{,}
\end{gather*}
where the implied contant is independent of\/~$N_0$ and~$\rho$.
\end{proposition}
\begin{proof}
Since the isotypic components of~$\pi_g$ have dimension at most one, the decomposition~\eqref{prop:invariants_via_Ga1_restriction:triple_product} implies that
\begin{align*}
 \dim\, \rmH^0\big(
 \rho_{N_0}^\vee \otimes \rho_{N_0}^\vee \otimes \rho
 \big)
&\le
  \sum_{\substack{
    g \in \Ga_1(N_0) \backslash \SL{2}(\ZZ) \slash \Ga_1(N_0) \\
    m \pmod{N}}}
  \dim\,\Res_{\Ga_1(N_0)} \Big(
  \big( \Res_{\Ga_1(N)}\, \rho \big)\isotypGaN{m}
  \Big)
\\
&=
  \# \big( \Ga_1(N_0) \backslash \SL{2}(\ZZ) \slash \Ga_1(N_0) \big)\,
  \dim\, \Res_{\Ga_1(N_0)}\, \Res_{\Ga_1(N)}\, \rho
\\
&=
  \# \big( \Ga_1(N_0) \backslash \SL{2}(\ZZ) \slash \Ga_1(N_0) \big)\,
  \dim(\rho)
\tx{.}
\end{align*}
To obtain the asymptotic upper bound, we may assume that~$N_0 > 4$. Using the Euler~$\varphi$\nbd function, we find that
\begin{gather*}
  \# \big( \Ga_1(N_0) \backslash \SL{2}(\ZZ) \slash \Ga_\infty \big)
=
  \mfrac{1}{2}
  \sum_{d \isdiv N_0}
  \varphi(d) \varphi(N_0 \slash d)
\tx{.}
\end{gather*}
Since~$\Ga_1(N_0)$ is generated by~$T$ and~$\Ga(N_0)$, we conclude that
\begin{gather*}
  \# \big( \Ga_1(N_0) \backslash \SL{2}(\ZZ) \slash \Ga_1(N_0) \big)
\ll
  \# \big( \Ga_1(N_0) \backslash \SL{2}(\ZZ) \slash \Ga_\infty \big)
\ll
  N_0\,
  \sum_{d \isdiv N_0} 1
\ll
  N_0^{1+\epsilon}
\tx{.}
\end{gather*}
\end{proof}

\subsection{Twisted permutation types}
\label{ssec:twisted_permutation_types}

In this section, we introduce the notion of twisted permutation representations, which generalize the notion of permutation representations. Their purpose in our work is twofold. On one hand, they allow for a more efficient calculation of the invariants in the domain of~\eqref{eq:thm:main_algorithm_preparation}. On the other hand, we apply this concept to the restriction of arithmetic types to the group generated by~$T \in \SL{2}(\ZZ)$. This allows for a dramatic reduction in size of the matrix in Algorithm~\ref{alg:main_algorithm}. The flexibility that we need, requires us to state the next definition for discrete groups as opposed to the special case of~$\SL{2}(\ZZ)$.

We will need the wreath product with the symmetric group~$\rmS_n$ on~$n$ letters, which for a group~$H$ and a positive integer~$n$ we define as as follows
\begin{gather*}
  H \wr \rmS_n
:=
  \big\{ (\pi,\ga) \,:\,
  \pi \in \rmS_n,\,
  \ga :\, \{1,\ldots,n\} \ra H
  \big\}
\tx{,}\quad
  (\pi',\ga') (\pi,\ga) 
:=
  \big( \pi' \pi,\, (\ga' \circ \pi) \cdot \ga \big)
=
  \big( \pi' \pi,\, i \mto \ga'(\pi(i)) \ga(i) \big)
\tx{.}
\end{gather*}
Elements of the form~$(\id, \ga)$ form a subgroup of~$H \wr \rmS_n$ that yields the quotient~$\rmS_n$.

Given a complex representation~$\sigma$ of~$H$, the wreath product gives rise to a representation~$\sigma \wr \rmS_n$ on~$\CC^n \times V(\sigma)$ that we define as
\begin{gather}
\label{eq:def:wreath_representation}
  \sigma \wr \rmS_n :\,
  H \wr \rmS_n
\lra
  \GL{}(\CC^n \times V(\sigma))
\tx{,}\;
  (\pi, \ga)
\lmto
  \big(
  (e_i \otimes w) \mto 
   e_{\pi(i)} \otimes \sigma(\ga(i)) w
  \big)
\tx{.}
\end{gather}

\begin{definition}
\label{def:twisted_perm_representation}
Let $G$ be a discrete group, $n$ a positive integer, and $\sigma$ a complex representation of some group~$H$. We call a complex representation~$\rho$ of~$G$ a twisted permutation representation of order $n$ with twist representation~$\sigma$ if~$\rho$ factors through~$\sigma \wr \rmS_n$. In other words, we have a group homomorphism~$\rho^\wr :\, G \ra \GL{}(W) \wr \rmS_n$ and an isomorphism~$\varphi_\rho$ between~$V(\rho)$ and~$V(\sigma \wr \rmS_n) = \CC^n \otimes V(\sigma)$ such that
\begin{gather*}
  \rho = \varphi_\rho^\ast \circ (\std(W) \wr \rmS_n) \circ \rho^\wr
\tx{.}
\end{gather*}
\end{definition}

We call~$\dim(V(\sigma))$ in Definition~\ref{def:twisted_perm_representation} the twist dimension of~$\rho$. If both~$\rho$ and~$\sigma$ in Definition~\ref{def:twisted_perm_representation} are arithmetic types, we call~$\rho$ a twisted permutation type.

Note that neither~$\rho^\wr$ nor~$\phi_\rho$ in Definition~\ref{def:twisted_perm_representation} are unique in general. Throughout this work, we identify twisted permutation representations~$\rho$ with a triple~$(\rho,\rho^\wr,\varphi_\rho)$. That is, we make an implicit choice of~$\rho^\wr$ and~$\varphi_\rho$.

In the next proposition, we describe twisted permutation representation in terms of induced representations. It can be viewed as a representation theoretic formulation of the orbit-stabilizer theorem from group theory. We need some notation to state Proposition~\ref{prop:decomp_orb_stab}. Given a twisted permutation representation~$\rho$ of order~$n$ with twist representation~$\sigma$, we define for~$I \subseteq \{1,\ldots,n\}$:
\begin{gather}
\label{eq:def:twisted_perm_representation:block}
  V(\rho)_I
=
  \linspan \{ \CC e_i \otimes V(\sigma) \,:\, i \in I \}
\subseteq
  V(\rho)
\tx{.}
\end{gather}
If~$I = \{ i \}$, we write~$V(\rho)_i$ for~\eqref{eq:def:twisted_perm_representation:block}. Note that~$V(\rho)_i \cong V(\sigma)$.

Composition of~$\rho^\wr$ with the projection~$H \wr \rmS_n \ra \rmS_n$ yields a map~$\rho^\wr_\pi :\, G \ra \rmS_n$. Inspecting the definition of~$\sigma \wr \rmS_n$ in~\eqref{eq:def:wreath_representation} and using that~$G$ is a group, we see that
\begin{gather}
\label{eq:def:twisted_perm_representation:block_stabilizer}
  \Stab_G\big( V(\rho)_I \big)
=
  \big\{ g \in G \,:\, \rho(g) V(\rho)_I \subseteq V(\rho)_I \big\}
=
  \big\{ g \in G \,:\, \rho(g) V(\rho)_I = V(\rho)_I \big\}
=
  \big( \rho^\wr_\pi \big)^{-1} \big( \Stab_{\rmS_n}(I) \big)
\subseteq
  G
\tx{.}
\end{gather}
Restricting~$\rho$ to this stabilizer yields a representation~$\rho_I$ on~$V(\rho)_I$:
\begin{gather}
\label{eq:def:twisted_perm_representation:block_representation}
  \rho_I :\,
  \Stab\big( V(\rho)_I \big)
\lra
  \GL{}\big( V(\rho)_I \big)
\tx{.}
\end{gather}
If~$I = \{i\}$ we write~$\rho_i$ for it.

\begin{proposition}
\label{prop:decomp_orb_stab}
Let~$\rho$ be a twisted permutation representation of a discrete group~$G$ of order~$n$. Further, let~$R \subseteq \{1,\ldots,n\}$ be a set of representatives for the action of~$\rho^\wr_\pi(G) \subseteq \rmS_n$. Then with~$\rho_i$ as in~\eqref{eq:def:twisted_perm_representation:block_representation}, we have an isomorphism
\begin{gather}
\label{eq:prop:decomp_orb_stab}
  \rho
\cong
  \bigoplus_{i \in R}
  \Ind_{\Stab(V(\rho)_i)}^{G} \rho_i
\tx{.}
\end{gather}
\end{proposition}
\begin{proof}
We can and will replace~$G$ by its image under~$\rho^\wr(G)$. This allows us to assume that~$G \subseteq H \wr \rmS_n$ and~$\rho$ is the restriction of~$\sigma \wr \rmS_n$ to~$G$. We also identify~$V(\rho)$ with~$\CC^n \otimes V(\sigma)$ via the isomorphism~$\varphi_\rho$ in Definition~\ref{def:twisted_perm_representation}. We write~$G_\pi \subseteq \rmS_n$ for the image of~$G$ under the quotient map from~$H \wr \rmS_n$ to~$\rmS_n$. We write~$G i$ for the orbit of~$1 \le i \le n$ under~$G_\pi$.

With~$R$ as in the statement of the proposition, we have the direct sum decomposition
\begin{gather*}
  V(\rho)
=
  \linspan \{ \CC e_j \,:\, 1 \le j \le n \} \otimes V(\sigma)
=
  \bigoplus_{i \in R}
  \linspan \{ e_j \,:\, j \in G i \} \otimes V(\sigma)
=
  \bigoplus_{i \in R}
  V(\rho)_{G i}
\tx{.}
\end{gather*}
It yields the isomorphism
\begin{gather*}
  \rho
\cong
  \bigoplus_{i \in R} \rho_{G i}
\tx{.}
\end{gather*}

We can and will replace~$G$ and~$\rho$ by~$\Stab(V(\rho)_{G i})$ and~$\rho_{G i}$ for any fixed~$i \in R$ to assume that~$G_\pi$ acts transitively on~$\{1,\ldots,n\}$ in the remainder of the proof. To further ease notation, we assume that~$R = \{1\}$ and set~$G_1 = \Stab(V(\rho)_1)$. We then have to show that
\begin{gather*}
  \rho \cong \Ind_{G_1}^G\,\rho_1
\tx{.}
\end{gather*}

Since~$G_\pi$ acts transitively on~$\{1,\ldots,n\}$, we can fix elements~$(\pi_i,\ga_i) \in G$ with~$\pi_i(1) = i$ for~$2 \le i \le n$, and let~$(\pi_1,\ga_1) \in G$ be the trivial element. We define a linear map by
\begin{gather*}
  \varphi :\,
  \CC^n \otimes V(\sigma)
\lra
  \CC[G] \otimes_{G_1} V(\rho)_1
\tx{,}\quad
  e_i \otimes w
\mto
  (\pi_i,\ga_i) \otimes \big( e_1 \otimes \sigma( \ga_i(1)^{-1} ) w \big)
\tx{.}
\end{gather*}
It is an isomorphism of vector spaces, since the~$(\pi_i,\ga_i)$ are coset representatives for the quotient of~$G$ by~$G_1$. We have to show that~$\varphi$ intertwines~$\rho$ and the induction of~$\rho_1$.

To this end, we fix~$1 \le i \le n$ and~$(\pi,\ga) \in G$, and set~$j = \pi(i)$. Note that the inverse of~$(\pi_j,\ga_j)$ equals~$(\pi_j^{-1}, \ga_j^{-1} \circ \pi_j^{-1})$, where the inverse of~$\ga_j$ is taken pointwise. For simplicity, we write the action of~$G$ via the induced representation of~$\rho_1$ as multiplication. Then we have to check that 
\begin{gather*}
  \varphi\big( \rho(\pi,\ga) (e_i \otimes w) \big)
=
  (\pi,\ga)\, \varphi\big( e_i \otimes w \big)
\tx{.}
\end{gather*}

We insert the action of~$\rho(\pi,\ga)$ using~$j = \pi(i)$ on the left hand side and then the definition of~$\varphi$ on both sides. This leads us to prove that
\begin{gather}
\label{eq:prop:decomp_orb_stab:final_intertwining_equality}
  (\pi_j,\ga_j) \otimes \big( e_1 \otimes \sigma\big( \ga_j(1)^{-1} \ga(i) \big) w \big)
=
  (\pi,\ga) (\pi_i,\ga_i) \otimes \big( e_1 \otimes \sigma( \ga_i(1)^{-1} ) w \big)
\tx{.}
\end{gather}
We now simplify the right hand side. Observe that~$\pi(\pi_i(1)) = j$, and thus~$(\pi_j,\ga_j)^{-1}\, (\pi,\ga)\, (\pi_i,\ga_i) \in G_1$.
We therefore have
\begin{gather}
\label{eq:prop:decomp_orb_stab:final_intertwining_equality_rhs}
  (\pi,\ga) (\pi_i,\ga_i) \otimes \big( e_1 \otimes \sigma(\ga_i(1)^{-1}) w \big)
=
  (\pi_j,\ga_j)
  \otimes
  \rho_1\big( (\pi_j,\ga_j)^{-1}\, (\pi,\ga)\, (\pi_i,\ga_i) \big)\,
  \big( e_1 \otimes \sigma(\ga_i(1)^{-1}) w \big)
\tx{.}
\end{gather}

We calculate the product in~$G_1$ that appears in the second tensor factor on the right hand side:
\begin{gather*}
  (\pi_j,\ga_j)^{-1}\, (\pi,\ga)\, (\pi_i,\ga_i)
=
  \big(
  \pi_j^{-1} \pi \pi_i,\,
  \ga_j^{-1} \circ \pi_j^{-1} \pi \pi_i \cdot
  \ga \circ \pi_i \cdot
  \ga_i
  \big)
\tx{.}
\end{gather*}
We insert this into the right hand side of~\eqref{eq:prop:decomp_orb_stab:final_intertwining_equality_rhs} and use the definition of~$\rho_1$, which is a restriction of~$\rho = \sigma \wr \rmS_n$, from~\eqref{eq:def:wreath_representation} to obtain
\begin{align*}
&
  (\pi_j,\ga_j)
  \otimes
  \rho_1\big( (\pi_j,\ga_j)^{-1}\, (\pi,\ga)\, (\pi_i,\ga_i) \big)\,
  \big( e_1 \otimes \sigma(\ga_i(1)^{-1}) w \big)
\\
={}&
  (\pi_j,\ga_j)
  \otimes
  \big( e_1
  \otimes
  \sigma\big( (
  \ga_j^{-1} \circ \pi_j^{-1} \pi \pi_i \cdot
  \ga \circ \pi_i \cdot
  \ga_i
  )(1) \big)\,
  \sigma(\ga_i(1)^{-1})
   w
   \big)
\\
={}&
  (\pi_j,\ga_j)
  \otimes
  \big( e_1
  \otimes
  \sigma\big(
  \ga_j^{-1}(1) \ga(i) \ga_i(1)
  \ga_i(1)^{-1}
  \big)
  w
  \big)
\tx{.}
\end{align*}
We recognize this as the left hand side of~\eqref{eq:prop:decomp_orb_stab:final_intertwining_equality_rhs}, and thus conclude the proof.
\end{proof}

\subsection{Translation orbits and Fourier expansions}
\label{ssec:translation_orbits_fourier_expansions}

We next specialize Proposition~\ref{prop:decomp_orb_stab} to the case of twisted permutation types~$\rho$ for~$\Ga_1(N)$ restricted to~$T^\ZZ \subseteq \Ga_1(N)$. The action of~$\rho(T)$ yields valuable information on the support of the Fourier expansions of modular forms of type~$\rho$.

Assuming that~$\rho(T)$ is diagonalizable, then we have~$\rmfe(\rho(T) f) = \rho(T) \rmfe(f)$. That is, $\rho(T)$ intertwines with the Fourier expansion map~$\rmfe$ in~\eqref{eq:def:fourier_expansion_map}. We equip~$\rmFE(\CC)$ with the~$T$\nbd action
\begin{gather}
\label{eq:fourier_expansion_taction}
  T\, \sum_{n \in \QQ} c(n) q^n
=
  \sum_{n \in \QQ} e(n) c(n) q^n
\tx{.}
\end{gather}
This action intertwines with the scalar-valued slash action of~$T$ on modular forms. The isotypic decomposition of~$\rmFE(\CC)$, using notation adopted from~\eqref{eq:Gamma1_isotypic}, is
\begin{gather}
\label{eq:fourier_expansion_isotypic_decomposition}
  \rmFE(\CC)
=
  \bigoplus_{n \in \QQ \slash \ZZ}
  \rmFE(\CC)[e(n)]
\tx{,}\quad
  \rmFE(\CC)[e(n)]
=
  q^n \CC \big\llbrkt q \big\rrbrkt \big[q^{-1}\big]
\tx{.}
\end{gather}

Combining this, we find that for~$f \in \rmM_k(\rho)$, we have
\begin{gather}
\label{eq:fourier_expansion_tinvariance}
  \rmfe(f)
=
  \rmfe\big( f \big|_{k,\rho}\,T \big)
=
  T\, \rmfe(f)
\tx{,}
\end{gather}
where~$T$ acts on the first tensor component of~$\rmFE(\CC) \otimes V(\rho)$ as in~\eqref{eq:fourier_expansion_taction} and on the second one by~$\rho(T)$. In other words, the Fourier expansion of modular forms is~$T$-invariant. This justifies the next corollary to Proposition~\ref{prop:decomp_orb_stab}.

\begin{corollary}
\label{cor:diag_tw_perm_taction}
Fix a twisted permutation type~$\rho$ for~$\Ga \subseteq \SL{2}(\ZZ)$ with~$T \in \Ga$. Let~$n$ be the twist order of~$\rho$ and~$\sigma :\, H \ra \GL{}(V(\sigma))$ its twist representation. Further, let~$R$ be a set of representatives for the action of\/~$T$ on~$\{1,\ldots,n\}$ via its image in~$H \wr \rmS_n$. Given~$1 \le i \le n$, write~$n_i = \# T^\ZZ i$ for the orbit length of~$i$. Define the representation~$\sigma_i$ of~$T^{n_i \ZZ}$ on~$V(\sigma)$ by
\begin{gather*}
  e_i \otimes \sigma_i\big( T^{n_i} \big) w
=
  \rho(T^{n_i}) ( e_i \otimes w )
\tx{.}
\end{gather*}

Then given a decomposition
\begin{gather*}
  \sigma_i
=
  \bigoplus_{s \in \QQ \slash n_i \ZZ} \sigma_i[e(s)]
\tx{,}\quad
  \sigma_i[e(s)]\big( T^{n_i} \big)
=
  e(s)
\tx{,}
\end{gather*}
into isotypic representations, in which only finitely many direct summands are nontrivial, we have the following decomposition of the restriction of~$\rho$ to~$T^\ZZ$ into isotypic representations:
\begin{gather*}
  \rho
=
  \bigoplus_{\substack{i \in R\\m \pmod{n_i}\\s \in \QQ \slash n_i \ZZ}}
  \Big(
  \CC
  \sum_{h \pmod{n_i}}
  e\big( \mfrac{-h (m + s)}{n_i} \big)
  \rho\big( T^h \big) \big( e_{i} \otimes \sigma_i[s] \big)
  \Big)
\tx{.}	
\end{gather*}
The action of\/~$T$ on the summand of index~$(i,m,s)$ is by~$e((m + s) \slash n_i)$.
\end{corollary}
\begin{proof}
Proposition~\ref{prop:decomp_orb_stab} reduces the proof to a computation of an isotypic decompositions of an induced representation. A calculation shows that the sums in the given decomposition of~$\rho$ are well-defined. Also the action of~$T$ follows from a direct calculation.
\end{proof}

\section{Implementation}%
\label{sec:implementation}

In this section, we briefly highlight two of the key ingredients of our implementation of Algorithm~\ref{alg:main_algorithm} and heuristically argue about its runtime. This section contains neither formal statements nor proofs, and thus takes the role of a commentary.

The discussion of the first two aspects of our implementation builds up on the results in Section~\ref{sec:decompositions}. The first one concerns the computation of invariants, which we will see in Section~\ref{ssec:time_complexity} is the slowest part of Algorithm~\ref{alg:main_algorithm}. The second one concerns the computation of the Fourier expansions, that we require in line~\ref{alg:row:main_algorithm:append_row} of Algorithm~\ref{alg:main_algorithm}. The major part of this section is dedicated to a heuristic analysis of our implementation's runtime performance. This is presented in Section~\ref{ssec:time_complexity}. In this section, we assume that fundamental arithmetic operations cost~$\rmO(1)$, which is correct for some base fields, but not for~$\QQab$. In Section~\ref{ssec:arithmetic_base_field}, we thus complement Section~\ref{ssec:time_complexity} with a discussion of the current state of base field arithmetic in our implementation.

In this section, we also discuss a variant of Algorithm~\ref{alg:main_algorithm} based on the decompositions
\begin{gather*}
  \rho_N
=
  \bigoplus_{\chi \pmod{N}} \rho_\chi
\quad\tx{and}\quad
  \cE_k(N)
=
  \bigoplus_{\chi \pmod{N}} \cE_k(\chi)
\tx{,}
\end{gather*}
where~$\chi$ runs through Dirichlet characters modulo~$N$, $\rho_\chi \hra \rho_N$ is as in~\eqref{eq:def:rhoN_rhochi}, and~$\cE_k(\chi) \subseteq \cE_k(N)$ is the corresponding subspace with~$\rho_\chi^\vee \thra \cE_k(\chi)$ as in Proposition~\ref{prop:abstract_eis}. It yields the direct sum decomposition
\begin{gather}
\label{eq:invariants_dirichlet_decomposition}
  \bigoplus_{\chi \pmod{N}}
  \rmH^0\big( \rho_\chi^\vee \otimes \rho,\, \QQab \big)
  \,\oplus\,
  \bigoplus_{\chi_1, \chi_2 \pmod{N_0}}
  \rmH^0\big( \rho_{\chi_1}^\vee \otimes \rho_{\chi_2}^\vee \otimes \rho,\, \QQab \big)
\end{gather}
that replaces the space of invariants in Algorithm~\ref{alg:main_algorithm}. The advantage of this decomposition lies in reducing the size of each representation for which we have to compute invariants on the left hand side of~\eqref{eq:def:thm:main_algorithm_preparation:phiE} at the expense of introducing further roots of unity in the Fourier expansion on its right hand side. When working numerically or modulo a suitable prime~$p$, the latter is hardly relevant for performance. A more detailed discussion of this issue can be found in Section~\ref{ssec:arithmetic_base_field}.

\subsection{Computation of invariants}
\label{ssec:computation_invariants}

The calculation of the invariants
\begin{gather*}
  \rmH^0\big( \rho_N^\vee \otimes \rho,\, \QQab \big)
  \oplus
  \rmH^0\big( \rho_{N_0}^\vee \otimes \rho_{N_0}^\vee \otimes \rho,\, \QQab \big)
\end{gather*}
in line~\ref{alg:row:main_algorithm:compute_invariants} of Algorithm~\ref{alg:main_algorithm} was left opaque. This kind of computation can be performed via GAP~\cite{gap-4-11-1}, and very efficiently so in general. We have found, however, that a custom made implementation that leverages specific features of the relevant arithmetic types achieves better performance. We would like to point out that, nevertheless, all calculations with permutation groups in our current implementation are invariably performed via GAP.\@

The motivation and starting point for our specialized implementation is the observation that~$\rho_N$ and~$\rho_\chi$, defined in~\eqref{eq:def:rhoN_rhochi}, are naturally twisted permutation types in the sense of Definition~\ref{def:twisted_perm_representation}. What is more, many of the arithmetic types~$\rho$ for which ones usually computes modular forms are twisted permutation types, too. Besides the cases~$\rho = \rho_N$ and~$\rho = \rho_\chi$, this includes inductions of Weil representations that correspond to Jacobi forms for subgroups of~$\SL{2}(\ZZ)$ by work of Skoruppa~\cite{skoruppa-2007}, or signed permutation representations that appear in generalized Moonshine~\cite{gaberdiel-persson-ronellenfitsch-volpato-2013,cheng-duncan-harvey-2014}.

Duals and tensor products of twisted permutation types are twisted permutation types, and twist orders are multiplicative under tensor products. In particular, if~$\rho$ is a twisted permutation type of twist order~$n$ and twist dimension~$d$, then
\begin{gather*}
  \rho_{N_0}^\vee \otimes \rho_{N_0}^\vee \otimes \rho
\quad\tx{and}\quad
  \rho_{\chi'_1}^\vee \otimes \rho_{\chi'_2}^\vee \otimes \rho
\end{gather*}
are twisted permutation types of twist dimension~$d$ and twist order~$\rmO(n N_0^{4+\epsilon})$ and~$\rmO(n N_0^{2+\epsilon})$. We can profit in the calculation of its invariants from the improved time complexity of the orbit-stabilizer algorithm as opposed to row echelon reduction. With~$\kappa \ge 2$ is the time complexity exponent of matrix multiplication explained in Section~\ref{ssec:time_complexity}, row echelon reduction would yield runtime~$\rmO((n N_0^{4 +\epsilon} )^\kappa )$ and~$\rmO((n N_0^{2 +\epsilon} )^\kappa )$. By Section~4.1 of~\cite{holt-eick-obrien-2005}, in our setting the orbit-stabilizer algorithm has runtime $\rmO(n^2 N_0^{8+\epsilon})$ and~$\rmO(n^2 N_0^{4+\epsilon})$, but by using stabilizer chains arising from the tensor product structure one can reduce this further to~$\rmO(n^2 N_0^{4+\epsilon})$ and~$\rmO(n^2 N_0^{2+\epsilon})$. Specifically, the runtime of the orbit-stabilizer algorithm receives a quadratic contribution from the size of the set acted on. In first case this set is the Cartesian product~$\Ga \slash \Ga_1(N_0) \times \Ga \slash \Ga_1(N_0)$ or~$\Ga \slash \Ga_0(N_0) \times \Ga \slash \Ga_0(N_0)$, and when using stabilizer chains it is~$\Ga \slash \Ga_1(N_0)$ or~$\Ga \slash \Ga_0(N_0)$. One combines this with usual linear algebra on the remaining~$d$ dimensions, which contributes~$\rmO(N_0^{3+\epsilon} d^2 + d^\kappa)$ to the runtime in the first case and~$\rmO(n N_0^{2+\epsilon} d^2 + d^\kappa)$ in the second case. Observe that the terms~$N_0^{3+\epsilon}$ and~$n N_0^{2+\epsilon}$ are connected to the number of generators of stabilizers resulting from the permutation group computation. This is not the number of generators guaranteed by the orbit-stabilizer algorithm, but results from an application of Farey fractions~\cite{kurth-long-2008,monien-2017-preprint} to the geometry of modular groups.

Our reformulation in terms of tensor products of twisted permutation types enables another major optimization. Stabilizer chains, which we mentioned before in passing, can be utilized in a similar way to the Schreier-Sim algorithm in Section~4.4.2 of~\cite{holt-eick-obrien-2005}. In practice, the required number of generators for the intermediate stabilizers that arise can be reduces drastically by the use of Farey fractions. Interestingly, Farey fractions also allow us to benefit from some of the advantages of the orbit-Schreier-vector algorithm in Section~4.4.1 of~\cite{holt-eick-obrien-2005}. In total, this seems to improves the runtime contribution of the permutation group computation to~$\rmO(n + N_0^{1 + \epsilon})$.

\subsection{Inflation and deflation via translation orbits}
\label{ssec:infl_defl_torbits}

In this section, we show how the decomposition in Section~\ref{ssec:translation_orbits_fourier_expansions} can be used to a priori bound the dimension of the space of Fourier expansions~$\rmfe_P(\rmM_k(\rho))$. This generalizes the concept of cusp expansions: When computing Fourier expansions of~$f |_k \ga$ for a modular form~$f \in \rmM_k(\Ga_0(N),\chi)$ and arbitrary~$\ga \in \SL{2}(\ZZ)$, then it suffices to consider representatives~$\ga$ of~$\Ga_0(N) \backslash \SL{2}(\ZZ) \slash \Ga_1(N)$, since the action of~$T \in \SL{2}(\ZZ)$ can be calculated on Fourier series.

We recall from Lemma~\ref{la:invariants_are_modular_forms} the inclusion
\begin{gather}
  \rmH^0\big( \cE_k(N) \otimes \rho \big)
  \,+\,
  \rmH^0\big( (\cE_l(N_0) \cdot \cE_{k-l}(N_0)) \otimes \rho \big)
\subseteq
  \rmM_k(\rho)
\tx{.}
\end{gather}
Fourier expansions that we handle are thus of the same level as~$\rho$, which impacts our implementation of Algorithm~\ref{alg:main_algorithm}. We call the resulting procedures deflation and inflation, adjusting language from the AbstractAlgebra package~\cite{fieker-hart-hofmann-johansson-2017} for power, Laurent, and Puiseux series.

Suppose that~$\rho$ has level~$N$ and~$\QQab$-structure as in Section~\ref{ssec:main_algorithm:algebraic_fourier_coefficients}. Then naively, we have
\begin{gather*}
  \rmfe\big( \rmM_k(\rho) \big)
\subseteq
  \CC \llbrkt q^{\frac{1}{N}} \rrbrkt
  \otimes
  V(\rho)
\tx{.}
\end{gather*}
Using merely this inclusion, the number of columns of the matrix~$M$ in Algorithm~\ref{alg:main_algorithm} is~$\rmO(N P \dim(\rho))$. Since we apply row echelon reduction to~$M$, the total contribution to the runtime is~$\rmO((N P \dim(\rho))^\kappa)$, where~$\kappa$ is the time complexity exponent explained in Section~\ref{ssec:time_complexity}.

In Section~\ref{ssec:translation_orbits_fourier_expansions}, however, we saw that the image of~$\rmM_k(\rho)$ under~$\rmfe_P$ is~$T$-invariant. We we will use that, since~$\rho$ is a congruence type, the transformation~$\rho(T)$ is diagonalizable and of finite order. Writing~$\rho_T$ for the restriction of~$\rho$ to the subgroup~$T^\ZZ \subset \SL{2}(\ZZ)$, the isotypic decomposition of~$\rmFE(\CC)$ given in~\eqref{eq:fourier_expansion_isotypic_decomposition} ensures that
\begin{gather*}
  \rmfe\big( \rmM_k(\rho) \big)
\subseteq
  \bigoplus_{\substack{s \in \frac{1}{N} \ZZ\\0 \le s < 1}}
  q^s \CC \llbrkt q \rrbrkt
  \otimes
  V\big( \rho_T[e(-s)] \big)
\tx{.}
\end{gather*}
In those cases in which this extends to Fourier expansions in~$\rmFE(\QQab) \otimes V(\rho,\QQab)$, it allows us to reduce the size of~$M$ in Algorithm~\ref{alg:main_algorithm} to~$\rmO(P \dim(\rho))$, thus removing the factor~$N$.

The purpose of Corollary~\ref{cor:diag_tw_perm_taction} is to make this reduction computationally efficient. We assume that~$\rho$ is a twisted permutation type, and for simplicity assume further that~$\rho = \sigma \wr \rmS_n$. We adopt notation from Corollary~\ref{cor:diag_tw_perm_taction}. In particular, we let~$R \subseteq \{1,\ldots,n\}$ be a set of representatives for the action arising from~$\rho(T^\ZZ)$. Then we have a deflation and inflation maps
\begin{alignat*}{3}
  \rmFE(\CC) \otimes \CC^n \otimes V(\sigma)
&{}\lra
  \rmFE(\CC) \otimes \CC^R \otimes V(\sigma)
\tx{,}\quad
&
  \sum_{i = 1}^n f_i
&{}\lmto
  \sum_{i \in R} f_i
\tx{,}\quad
&&
  f_i \in \rmFE(\CC) \otimes \CC e_i \otimes V(\sigma)
\\
  \rmFE(\CC) \otimes \CC^R \otimes V(\sigma)
&{}\lra
  \rmFE(\CC) \otimes \CC^n \otimes V(\sigma)
\tx{,}\quad
&
  \sum_{i \in R} f_i
&{}\lmto
  \sum_{\substack{i \in R\\h \pmod{\# T^\ZZ i}}} T\,f_i
\tx{,}\quad
&&
  f_i \in \rmFE(\CC) \otimes \CC e_i \otimes V(\sigma)
\tx{,}
\end{alignat*}
where~$\# T^\ZZ i$ is the length of the orbit of~$i$ under~$T^\ZZ$. In general, deflation is left inverse to inflation. On the subspace of~$T$-invariant Fourier expansions, that is, for Fourier expansions of modular forms, deflation is injective, and inflation is an inverse to it.

\subsection{Arithmetic operations in the base fields}
\label{ssec:arithmetic_base_field}

In Section~\ref{ssec:time_complexity} we make the assumption that each fundamental arithmetic operation contributes~$\rmO(1)$ to the runtime of our algorithm. This is justified for calculations over finite fields and for numerical calculations of bounded precision at a limited number of infinite places. However, for exact arithmetic over~$\QQab$ and cyclotomic fields of arbitrary degree this is false, which on its own is a reason why the current implementation of our algorithm falls behind alternative ones. This issue gains further importance in situations where intermediate fields of higher degree appear in the calculation, as is the case in the variant of Algorithm~\ref{alg:main_algorithm} discussed in the beginning of Section~\ref{sec:implementation}.

The type systems provided by Julia and Nemo/Hecke allowed us to provide an implementation that is polymorphic in the type of the base ring elements. In addition to the standard requirements for fields in Nemo/Hecke, our implementation makes use of a cyclotomic tower over the base field that we provide for cyclotomic fields, for finite fields, and for~$\CC$ modeled by ball arithmetic (which is an approximate field and thus requires some extra steps to accommodate). Our implementation thus is sufficiently flexible to accommodate advanced methods like multi-modular calculations. From a practical perspective, consider the performance of matrix operations on, say,
\begin{gather*}
  \Mat{m,n}\big( \bigoplus_{q \in \cQ} \bbF_q \big)
\quad\tx{and}\quad
  \Mat{m,n}(\ZZ) \otimes \bigoplus_{q \in \cQ} \bbF_q
\end{gather*}
for a finite set~$\cQ$ of powers of mutually distinct primes. The former fits into the polymorphic implementation, but suffers from degraded cache locality. In future work when we apply our implementation, we plan to implement the latter variant.

An implementation over~$\bigoplus_{q \in \cQ} \bbF_q$ does not suffice for a complete multi-modular implementation, which requires height bounds. For clarity and comparison observe that Cohen's implementation of scalar-valued modular forms, which we discuss in Section~\ref{ssec:product_scalar_valued_eisenstein_series}, does employ multi-modular arithmetic to calculate cusp expansions. He is able to derive height bounds from work of Katz~\cite{katz-1973} and Deligne-Rapoport~\cite{deligne-rapoport-1973}, since he first determines Fourier expansions at~$\infty$ via a trace formula. In our algorithm we do not have access to a priori Fourier expansions of any component of our arithmetic type~$\rho$, and the work of Deligne-Rapoport would not apply directly, either. On the other hand, when naively deriving a height bound from the products of Eisenstein series, we obtain a prohibitively large bound.

We are thus naturally led to an approach that incorporates multi-modular calculations with calculations at the infinite places. We have not made any attempt to implement our algorithm using machine-size floating point arithmetic, but rely on ball arithmetic provided by Arb~\cite{johansson-2017} to track rounding errors. At the time of writing, a multi-modular variant of our implementation is not yet functional, but we intend to further evaluate it as it seems one of the more promising ways to perform Fourier expansions over large cyclotomic fields that appear when determining cusp expansions of modular forms of large level.

\subsection{Heuristic discussion of time complexity}
\label{ssec:time_complexity}

We next discuss the time complexity of computing a basis of~$\rmM_k(\rho)$ via Algorithm~\ref{alg:main_algorithm} and its modification arising from~\eqref{eq:invariants_dirichlet_decomposition}. Throughout the section, we assume that~$\rho$ is a congruence type of level~$N$ for~$\Ga \subseteq \SL{2}(\ZZ)$ with~$\Ga_1(N) \subseteq \Ga$ and that~$\rho(T)$ is diagonalizable. Note that any arithmetic type can be considered as a twisted permutation type of twist order one. We therefore can and will view~$\rho$ as a twisted permutation type of twist order~$n$ and twist dimension~$d$.

We assume also that fundamental arithmetic operations cost~$\rmO(1)$, which covers numerical calculations and calculations modulo primes~$p$ in a fixed range, but is incorrect for computations in number fields. The time complexity of Algorithm~\ref{alg:main_algorithm} depends on the time complexity of various matrix operations. Given~$n \times n$ matrices~$A$ and~$B$, we assume that the time complexity of computing the product~$A B$, the inverse~$A^{-1}$, and of diagonalizing~$A$ is~$\rmO(n^{\kappa+\epsilon})$, where here and in the remainder of this section~$\kappa$ denotes the exponent in the time complexity of matrix multiplication~\cite{demmel_et_al-2007}. We have~$\kappa \lessapprox 2.373$ in theory~\cite{alman_et_al-2021}, but in practice implementations rely on results of Strassen~\cite{strassen-1969}, whose algorithm established that~$\kappa \lessapprox 2.807$. 

\paragraph{Translation orbits}

We estimate the time complexity of computing orbits under~$T$ that appear in Corollary~\ref{cor:diag_tw_perm_taction} and Section~\ref{ssec:infl_defl_torbits}. Let~$n$ be the twist order of~$\rho$ and~$\sigma$ its twist representation of dimension~$d$ for an auxiliary group~$H$.

Per Definition~\ref{def:twisted_perm_representation}, we have an associated group homomorphism~$\rho^\wr$ from~$\Ga$ to~$H \wr \rmS_n$. It yields an action of~$T \in \Ga$ on~$\{1,\ldots,n\}$ via the projection to~$\rmS_n$. Applying the orbit algorithm in this setting has time complexity~$\rmO(n)$. The stabilizer of any~$1 \le i \le n$ is generated by~$T^{n_i}$, where~$n_i := \# T^\ZZ i$ is the orbit length of~$i$. 

In typical situations, the action of~$\rho(T^{n_i})$ on the~$i$\thdash\ component~$V(\rho)_i$ of~$\CC^n \otimes V(\sigma)$ (cf.~\eqref{eq:def:twisted_perm_representation:block}) is in diagonal form already. This includes the cases of arithmetic types associated with Dirichlet characters and of Weil representations. In general, diagonalizing it contributes~$\rmO(\# R d^\kappa)$ to the total time complexity, where~$R$ is a set of orbit representatives for~$T^\ZZ$ acting on~$\{1,\ldots,n\}$.

Given orbits for~$T^\ZZ$ on~$\{1,\ldots,n\}$ and an isotypic decomposition of~$\rho(T^{n_i \ZZ})$ acting on~$V(\rho)_i$, Corollary~\ref{cor:diag_tw_perm_taction} provides an isotypic decomposition of~$\rho(T)$. The time complexity contributed by this step equals $\sum_{i \in R} \rmO\big( n_i^2 d \big) \le \rmO(n^2 d)$.

In summary, we obtain an isotypic decomposition of~$\rho(T)$ including a convenient basis in
\begin{gather*}
  \rmO(n) + \rmO(\# R d^\kappa)
  \,+\,
  \rmO(n^2 d)
\;=\;
  \rmO\big( \dim(\rho) (d^{\kappa-1} + n) \big)
\tx{.}
\end{gather*}

\paragraph{Invariants}

We discuss the contribution of~\eqref{prop:invariants_via_Ga1_restriction:triple_product} to the time complexity, and neglect the one of~\eqref{prop:invariants_via_Ga1_restriction:double_product}, which is dominated by the former. 

We recall from Proposition~\ref{prop:triple_product_dim_bound} that the space of invariants on the right hand side of~\eqref{prop:invariants_via_Ga1_restriction:triple_product} is bounded by $\mathcal{O}( N_0^{1+\epsilon} \dim(\rho) )$. The estimate in Proposition~\ref{prop:triple_product_dim_bound} arises from a bound on~$\Ga_1(N_0)$-invariant vectors. We apply Frobenius reciprocity as in~\eqref{eq:frobenius_reciprocity} to reconstruct invariants for~$\Ga$. This results in a basis of invariants in~\eqref{prop:invariants_via_Ga1_restriction:triple_product}, each of whose elements is a sum over less than~$[\Ga : \Ga_1(N_0)]$ vectors of dimension~$\dim(\rho)$. The time complexity thus depends on the index of~$\Ga_1(N_0)$ in~$\Ga$. We obtain a total contribution of
\begin{gather*}
  \rmO\big( N_0^{1+\epsilon} \dim(\rho)\, [\Ga : \Ga_1(N_0)] \dim(\rho) \big)
=
  \rmO\big( N_0^{1+\epsilon}\, [\Ga : \Ga_1(N_0)] \dim(\rho)^2 \big)
\tx{.}
\end{gather*}

For later purposes, we also record that for each of these invariant basis elements each component of~$\rho$ on average can be expressed by the following number of terms:
\begin{gather*}
  \rmO\big( [\Ga : \Ga_1(N_0)] \dim(\rho) \slash \dim(\rho) \big)
=
  \rmO\big( [\Ga : \Ga_1(N_0)] \big)
\tx{.}
\end{gather*}

\begin{remark}
\label{rm:time_complexity_invariants}
If~$N_0 \gg k$, which is usually the case, the estimate~$N_0^{1+\epsilon} \dim(\rho)$ is larger than the dimension of~$\rmM_k(\rho)$, which is of order~$\rmO(k \dim(\rho))$. Therefore we expect significant redundancy among the modular forms associated to the invariants in~\eqref{prop:invariants_via_Ga1_restriction:triple_product}. This is the ultimate reason for the preemptive comparison of dimensions in line~\ref{alg:row:main_algorithm:dimension_comparison} of Algorithm~\ref{alg:main_algorithm}. In practice, the required number of invariants for small and fixed~$k$, appears be of order~$\rmO( N^\epsilon \dim(\rho) )$. Assuming this, we can replace the total contribution of the calculation of invariants by
\begin{gather*}
  \rmO\big( N_0^\epsilon\, [\Ga_1(N_0) : \Ga] \dim(\rho)^2 \big)
\tx{.}
\end{gather*}
\end{remark}

\paragraph{Fourier expansions}

Given a basis of invariants as in line~\ref{alg:row:main_algorithm:compute_invariants} of Algorithm~\ref{alg:main_algorithm}, the next step is to compute Fourier expansions. Naively, for each basis element this involves the computation of~$\dim(\rho)$ sums of products of Eisenstein series each with~$P N$ coefficients, where~$P$ is the truncation bound in Theorem~\ref{thm:main_algorithm} and Algorithm~\ref{alg:main_algorithm}. Each sum involves~$\rmO( [\Ga : \Ga_1(N_0)])$ products on average. A typical choice of~$P$ is the Sturm bound, which is of size~$\rmO(k [\SL{2}(\ZZ) : \Ga])$. Products of power series of length~$P N$ can be determined in time complexity~$\rmO(P N \log(B N))$. The naive contribution to time complexity of Fourier expansions of an invariant is hence
\begin{gather*}
  \rmO\big(
  \dim(\rho) [\Ga : \Ga_1(N_0)]\,
  (k [\SL{2}(\ZZ) : \Ga] N)^{1 + \epsilon}
  \big)
\ll
  \rmO\big( (k N_0^2 N)^{1 + \epsilon}\, \dim(\rho) \big)
\tx{.}
\end{gather*}

Now we note that we can employ the~$T$-invariance of Fourier expansion of modular forms in~\eqref{eq:fourier_expansion_tinvariance}. This means that we can reduce ourselves to the computation with power series of length~$P$, removing the contribution of~$N$ to the time complexity. This yields time complexity
\begin{gather*}
  \rmO\big( (k N_0^2)^{1 + \epsilon} \dim(\rho) \big)
\tx{.}
\end{gather*}

\paragraph{Reduced row echelon forms}

The truncated and deflated vector-valued Fourier series expansions are gathered in a matrix in line~\ref{alg:row:main_algorithm:append_row} of Algorithm~\ref{alg:main_algorithm}. We provide a pessimistic bound for the time complexity of the row echelon reduction in line~\ref{alg:row:main_algorithm:row_echelon_reduction} by ignoring the extra termination condition on line~\ref{alg:row:main_algorithm:dimension_comparison}.

With the discussion in Section~\ref{ssec:infl_defl_torbits} in mind, we can estimate the size of the matrix~$M$ in Algorithm~\ref{alg:main_algorithm} by
\begin{gather*}
  \rmO\big( N_0^{1+\epsilon} \dim(\rho) \big)
  \,\cdot\,
  \rmO\big( k [\SL{2}(\ZZ) : \Ga] \dim(\rho) \big)
\tx{.}
\end{gather*}
The number of rows arises from the dimension of invariants, and the number columns from the Fourier expansion. This matrix is reduced to row echelon form. The time complexity contributed by this step is
\begin{gather*}
  \rmO\big( \max\big\{
  N_0^{1+\epsilon} \dim(\rho),\,
  k [\SL{2}(\ZZ) : \Ga] \dim(\rho)
  \big\}^\kappa \big)
\tx{.}
\end{gather*}

By Remark~\ref{rm:time_complexity_invariants}, for small and fixed~$k$ we expect that
preemptive termination in line~\ref{alg:row:main_algorithm:dimension_comparison} of Algorithm~\ref{alg:main_algorithm} yields the significantly better estimate
\begin{gather*}
  \rmO\big( \max\big\{
  N_0^\epsilon \dim(\rho),\,
  [\SL{2}(\ZZ) : \Ga] \dim(\rho)
  \big\}^\kappa \big)
\tx{.}
\end{gather*}

\paragraph{Fourier expansions of arbitrary precision}

Given a basis of modular forms computed along the lines of Algorithm~\ref{alg:main_algorithm}, we also obtain expressions for their components in terms of products of Eisenstein series as mentioned in~\ref{rm:main_algorithm:symbolic_expressions}. The length of each these expressions, in the shape we implemented it, is tightly connected to the explicit Frobenius reciprocity in~\eqref{eq:frobenius_reciprocity}, which yields the length estimate~$\rmO([\Ga : \Ga_1(N_0)])$. As before, products of Eisenstein series to precision~$P$ can be computed in time complexity~$\rmO(P \log P)$. In total, computing a single component has time complexity
\begin{gather*}
  \rmO\big( [\Ga : \Ga_1(N_0)]\, P \log P \big)
\tx{.}
\end{gather*}

\paragraph{Modular forms for~$\Ga_0(N)$}

We conclude with the special case of modular forms for~$\Ga_0(N)$ and a Dirichlet character~$\chi$. If we determine merely the cusp expansion at~$\infty$ for small and fixed~$k$, then~$\rho \cong \chi$ is one-dimensional. We assume that~$N_0 \approx N$, which is experimentally true for small~$k$. Then the heuristic time complexity estimates for the computation of~$T$-invariants, of invariants, the Fourier expansion of all invariants, and the reduced row echelon normal form are $\rmO(1)$, $\rmO(N^{3+\epsilon})$, $\rmO(N^{3+\epsilon})$, and~$\rmO(N^{(1+\epsilon)\kappa})$. The computation of invariants and their Fourier expansions is dominant. It yields the total time complexity
\begin{gather*}
  \rmO\big( N^{3+\epsilon} \big)
\tx{.}
\end{gather*}

Under the experimental assumption from Remark~\ref{rm:time_complexity_invariants}, the computation of invariants becomes fast and the row echelon form is dominant. It yields a total time complexity
\begin{gather*}
  \rmO\big( N^{\kappa + \epsilon} \big)
\tx{.}
\end{gather*}

We conclude with a remark on the observed performance. Also in case of classical modular forms for~$\Ga_0(N)$ our implementation is surprisingly slow. Profiling reveals that it is dominated by the calculation of invariants which happens in GAP and via an implementation of Farey fractions that improves upon the one provided in Sage~\cite{sage-9-2}.

\subsection{Examples}

To illustrate the use of the implementation provided by the authors, we consider two cases that we mentioned in the introduction: Modular forms for congruence subgroups of non-split Cartan type and modular forms for ``Moonshine-like'' arithmetic types.

\paragraph{Congruence subgroups of non-split Cartan type}

We recall from the introduction that modular forms for the congruence subgroups~$\Ga_{\mathrm{ns}}(N)$ of non-split Cartan type are usually computed via cusp expansions of modular forms for~$\Ga_1(N^2)$~\cite{mercuri-schoof-2020,assaf-2020-preprint}. In our framework, we can treat them on equal footing with modular forms for any other congrunce subgroup when taking the induction of the trivial type of~$\Ga_{\mathrm{ns}}(N)$ to~$\SL{2}(\ZZ)$. The computation of a $20$-dimensional such space is illustrated in Listing~\ref{lst:nonsplit}.

\begin{lstlisting}%
[caption={Julia code to compute modular forms for~$\Ga_{\mathrm{ns}}(7)$ of weight~$6$},
 label=lst:nonsplit]
using ModularForms
rho = induction(SL2Z, TrivialArithmeticType(GammaNS(7)))
mfs = ModularFormsSpace(6, rho, QQab)
[fourier_expansion(f,2) for f in basis(mfs)]
\end{lstlisting}

The resulting Fourier expansions are too long to reproduce in print. As an example we provide the Fourier expansion of the first component of the first basis element, which is associated with a modular form for~$\Gamma_{\mathrm{ns}}(7)$.
\begin{align*}
&
1
+
\big( -5 \zeta_{42}^{11} - \zeta_{42}^{9} - 2 \zeta_{42}^{8} + \zeta_{42}^{6} + 5 \zeta_{42}^{4} + 2 \zeta_{42} + 2 \big) q^{\frac{4}{7}}
\\&
+
\big( 9 \zeta_{42}^{11} + 7 \zeta_{42}^{9} + 7 \zeta_{42}^{8} + 9 \zeta_{42}^{6} - 9 \zeta_{42}^{4} + 5 \zeta_{42}^{3} - 7 \zeta_{42} - 5 \big) q^{\frac{5}{7}}
\\&
+
\big( -19 \zeta_{42}^{11} + 16 \zeta_{42}^{9} + 19 \zeta_{42}^{8} - 23 \zeta_{42}^{6} + 19 \zeta_{42}^{4} + 16 \zeta_{42}^{3} - 19 \zeta_{42} \big) q^{\frac{6}{7}}
\\&
+
\big( 39 \zeta_{42}^{11} + 10 \zeta_{42}^{8} - 10 \zeta_{42}^{6} - 39 \zeta_{42}^{4} - 39 \zeta_{42}^{3} - 10 \zeta_{42} - 74 \big) q
\\&
+
\big( -60 \zeta_{42}^{9} + 17 \zeta_{42}^{8} - 17 \zeta_{42}^{6} + 17 \zeta_{42}^{3} - 17 \zeta_{42} + 60 \big) q^{\frac{8}{7}}
\\&
+
\big( 32 \zeta_{42}^{11} - 159 \zeta_{42}^{9} - 253 \zeta_{42}^{8} - 32 \zeta_{42}^{4} - 253 \zeta_{42}^{3} + 253 \zeta_{42} + 32 \big) q^{\frac{9}{7}}
\\&
+
\big( 49 \zeta_{42}^{11} - 49 \zeta_{42}^{9} + 206 \zeta_{42}^{6} - 49 \zeta_{42}^{4} + 249 \zeta_{42}^{3} + 206 \big) q^{\frac{10}{7}}
\\&
+
\big( -465 \zeta_{42}^{11} + 117 \zeta_{42}^{9} - 240 \zeta_{42}^{8} - 117 \zeta_{42}^{6} + 465 \zeta_{42}^{4} + 240 \zeta_{42} + 240 \big) q^{\frac{11}{7}}
\\&
+
\big( 346 \zeta_{42}^{11} + 704 \zeta_{42}^{9} + 704 \zeta_{42}^{8} + 346 \zeta_{42}^{6} - 346 \zeta_{42}^{4} + 582 \zeta_{42}^{3} - 704 \zeta_{42} - 582 \big) q^{\frac{12}{7}}
\\&
+
\big( -875 \zeta_{42}^{11} + 594 \zeta_{42}^{9} + 875 \zeta_{42}^{8} - 731 \zeta_{42}^{6} + 875 \zeta_{42}^{4} + 594 \zeta_{42}^{3} - 875 \zeta_{42} \big) q^{\frac{13}{7}}
+
O(q^2)
\end{align*}

To illustrate the performance of our implementation, we consider the induction of the trivial type from~$\Ga_{\mathrm{ns}}(3)$, $\Ga_{\mathrm{ns}}(5)$, and~$\Ga_{\mathrm{ns}}(7)$ to~$\SL{2}(\ZZ)$. These representations have dimensions~$6$, $20$, and~$42$. The calculation of the associated weight-$6$ modular forms takes~$0.08$, $0.48$, and~$5.28$ seconds, respectively, on a single hardware thread of a home computer with AMD Ryzen 9 3900XT. Profiling the calculation reveals that the vast majority of runtime is spend on determining invariants on the left hand side of~\eqref{eq:thm:main_algorithm_preparation}. This is despite the fact that linear algebra is performed over the cyclotomic field of order~$42$. In more detail, most time is spent in orbit calculations using the highly optimized GAP implementation~\cite{gap-4-11-1}, and in our implementation of Farey symbols to determine generators of stabilizer subgroups. The latter is a variant of the implementation in Sage~\cite{sage-9-2} going back to the work of Kurth--Long~\cite{kurth-long-2008}. Our implementation employs more suitable data structures and thus achieves a 10\nbd{}fold speedup in a benchmark with~$\Ga(64)$, that now requires~$37$~seconds to complete.

\paragraph{Moonshine-like arithmetic types}

To illustrate the use of our implementation in the context of Moonshine, we intended to examine generalized Mathieu Moonshine~\cite{gaberdiel-persson-ronellenfitsch-volpato-2013}. However, we were unable to verify that twists computed using the code accompanying the paper by Gaberdiel--Persson--Ronellen\-fitsch--Volpato yield representations of~$\SL{2}(\ZZ)$. Instead, we have implemented the following construction that mimics the setup in any kind of generalized Moonshine: Given a finite abelian group~$G$, we have a permutation action of~$\SL{2}(\ZZ)$ from the right on~$G \times G$. We determine all twisted permutation types that extend this permutation action, pick a random one, and compute a corresponding space of modular forms. This is illustrated in Listing~\ref{lst:moonshine-like}. We execute these calculations over a number field and also determine the corresponding space of modular forms over~$\bbF_{17}$.

\begin{lstlisting}%
[caption={Julia code to compute modular forms for a Moonshine-like type of weight~$4$},
 label=lst:moonshine-like]
using ModularForms
using ModularForms.MoonshineLikeType
rho = rand_moonshine_like_type([2,2], 4)
mfsQQab = ModularFormsSpace(4, rho, QQab)
mfsFF = ModularFormsSpace(4, rho, FiniteField(17)[1])
\end{lstlisting}

The runtime of calculations with Moonshine-like types varies vastly with the twist. If~$G = \ZZ \slash 2 \times \ZZ \slash 2$, which is a group that appears in Mathieu Moonshine, there are~$2^{41}$ different twists of the permutation action on~$G \times G$. Table~\ref{tab:moonshine-like-example} provides the permutation and the twists for a typical example. We label the basis elements by integers between~$1$ and~$16$ and give their images including the associated twists under two generators of~$\SL{2}(\ZZ)$. The calculation of weight-$4$ modular forms for this type over a number field takes~1.15 seconds and~1.14 seconds over~$\bbF_{17}$. The marginal difference between these runtimes illustrates the large proportion that it takes to calculate the left hand side of~\eqref{eq:thm:main_algorithm_preparation}.

\begin{table}[H]
\caption{An example of a Moonshine-like twisted permutation action}
\label{tab:moonshine-like-example}
\begin{center}
\begin{tabular}{l*{16}{c}}
\toprule
 & 1 & 2 & 3 & 4 & 5 & 6 & 7 & 8 & 9 & 10 & 11 & 12 & 13 & 14 & 15 & 16
\\\midrule
   \multirow{2}{*}{$\begin{psmatrix} 0 & -1 \\ 1 & 0 \end{psmatrix}$}
 & $1$ & $5$ & $9$ & $13$ & $2$ & $6$ & $10$ & $14$ & $3$ & $7$ & $11$ & $15$ & $4$ & $8$ & $12$ & $16$
\\
 & $\zeta_4$ & $\zeta_2$ & $\zeta_4^3$ & $\zeta_4^3$ & $1$ & $\zeta_4$ & $\zeta_2$ & $\zeta_4$
 & $\zeta_4$ & $1$ & $\zeta_1$ & $\zeta_4$ & $\zeta_4^3$ & $\zeta_4$ & $\zeta_4$ & $\zeta_4$
\\\midrule
   \multirow{2}{*}{$\begin{psmatrix} 1 & 1 \\ 0 & 1 \end{psmatrix}$}
 & $1$ & $2$ & $3$ & $4$ & $6$ & $5$ & $8$ & $7$ & $11$ & $12$ & $9$ & $10$ & $16$ & $15$ & $14$ & $13$
\\
 & $\zeta_4$ & $1$ & $\zeta_2$ & $\zeta_4^3$ & $\zeta_4$ & $\zeta_2$ & $\zeta_4^3$ & $\zeta_4$ &
   $\zeta_4^3$ & $1$ & $\zeta_4^3$ & $1$ & $\zeta_2$ & $\zeta_4$ & $\zeta_4$ & $\zeta_2$
\\
\bottomrule
\end{tabular}
\end{center}
\end{table}

\section{Comparisons with alternative algorithms}
\label{sec:alternative_algorithms}

There are two established algorithms to compute classical modular forms, which were surveyed in~\cite{best_et_al-2020}: Modular symbols and the Eichler-Selberg trace formula. Both are a priori limited to the computation of scalar-valued modular forms for congruences subgroups~$\Ga$ via their Hecke eigenvalues, that is, their cusp expansions at~$\infty$. If the level of~$\Ga$ is square-free the associated vector-valued modular forms of type~$\Ind_{\Ga}\,\bbone$ for~$\SL{2}(\ZZ)$ can be recovered from the action of Atkin-Lehner involutions. For general~$\Ga$ this is not possible. Fixing the weight, as we did in Section~\ref{ssec:time_complexity}, computing the Fourier expansions up to precision~$P \gg N$ of a basis of~$\rmM_k(\Ga_0(N), \chi)$ using modular symbols or the trace formula has time complexity
\begin{gather*}
  \rmO\big( N^{1+\epsilon} P^2 \big)
\quad\tx{and}\quad
  \rmO\big( N^{\frac{3}{2}+\epsilon} P^{\frac{3}{2}} \big)
\tx{.}
\end{gather*}
While both algorithms compute less than what Algorithm~\ref{alg:main_algorithm} does, they do so significantly faster. 

There are two further algorithms in the literature that target vector-valued modular forms and share some features with the our Algorithm~\ref{alg:main_algorithm}. The first one by Cohen~\cite{cohen-2019} is based on products of scalar-valued Eisenstein series for~$\Ga_0(N)$ with character. The second one by Williams~\cite{williams-2018a} is based on products of Eisenstein series for the Weil representation with theta series of weight~$\frac{1}{2}$, which are residues of Eisenstein series. Both are restricted to specific kinds of vector-valued modular forms, but when applicable perform significantly better than our implementation of Algorithm~\ref{alg:main_algorithm}. We will discuss the reason why and potential remedies to their limitations.

\subsection{Products of scalar-valued Eisenstein series}%
\label{ssec:product_scalar_valued_eisenstein_series}

Cohen~\cite{cohen-2019} implemented an algorithm in Pari/GP to calculate Petersson scalar products of modular forms in~$\rmM_k(\Ga_0(N),\chi)$, the space of modular forms of weight~$k$ with~$f |_k \ga = \chi(d) f$ for~$\ga = \begin{psmatrix} a & b \\ c & d \end{psmatrix} \in \Ga_0(N)$ for a Dirichlet character~$\chi$. As one of the steps in his algorithm he requires the cusp expansions of modular forms. He determines these cusp expansions after having calculated modular forms via Eichler-Selberg trace formula~\cite{belabas-cohen-2018}, but we remark that his approach can be adjusted to calculate a basis for~$\rmM_k(\Ga_0(N),\chi)$. General cusp expansions enable the computation of a basis of~$\rmM_k(\rho_\chi)$, where~$\rho_\chi$ is as in~\eqref{eq:def:rhoN_rhochi}. In other words, one can recover from Cohen's work the special case~$\rho = \rho_\chi$ of Algorithm~\ref{alg:main_algorithm}.

\begin{remark}
Any irreducible congruence type with nontrivial~$T$-fixed vectors embeds into a suitable~$\rho_\chi$. To make general congruence types accessible to Cohen's approach, one can use the vector-valued Hecke operators~$\rmT_N$ of~\cite{raum-2017} and the inclusion~$\rho \hra \rmT_N\, \rmT_N\, \rho$. If~$N$ is suitably chosen, then there is a subrepresentation of~$\rho' \hra \rmT_N\,\rho$ generated by its $T$\nbd fixed vectors such that~$\rho \hra \rmT_N\, \rho'$.
\end{remark}

The theoretical foundation of Cohen's algorithm was given by Borisov--Gunells in their work on toric modular forms~\cite{borisov-gunnells-2001,borisov-gunnells-2001b,borisov-gunnells-2003}. As a corollary to their results one finds that if~$k > 2$, then
\begin{gather*}
  \rmM_k(\Ga_0(N), \chi)
\;\subseteq\;
  \rmE_k(\Ga_0(N), \chi)
  +
  \sum_{l = 1}^{k-1}
  \rmE_l(\Ga_1(N))
  \cdot
  \rmE_{k-l}(\Ga_1(N))
\tx{,}
\end{gather*}
where~$\rmE_k(\Ga_0(N), \chi)$ stands for the spaces of weight-$k$ Eisenstein series for~$\Ga_0(N)$ and~$\chi$. In other words, we have
\begin{gather}
\label{eq:borisov_gunnells}
  \rmM_k(\Ga_0(N), \chi)
\;=\;
  \rmE_k(\Ga_0(N), \chi)
  +
  \sum_{l = 1}^{k-1}
  \rmH^0 \big( 
  (\rmE_l(\Ga_1(N)) \cdot \rmE_{k-l}(\Ga_1(N)))
  \otimes
  \chi
  \big)
\tx{,}
\end{gather}
where we allows ourselves to identify the invariants on the right hand side with the corresponding modular forms. If~$k = 2$, there is no direct analogue of~\eqref{eq:borisov_gunnells}, but Cohen works around this by a variant of it.

The equality in~\eqref{eq:borisov_gunnells} parallels the isomorphism in~\eqref{eq:raum_xia}, but the underlying representation theory is much simpler, since the former features~$\rmE_l(\Ga_1(N))$ as opposed to~$\cE_l(N) = \rmE_l(\Ga(N))$ from Section~\ref{ssec:prelim:eisenstein}. Recall from Proposition~\ref{prop:abstract_eis} that~$\cE_l(N)$ as a representation of~$\SL{2}(\ZZ)$ is a quotient of the induced representation~$\rho_N$ from~\eqref{eq:def:rhoN_rhochi}. By Mackey's Double Coset Theorem, its restriction to~$\Ga_0(N)$ in general contains multidimensional, irreducible representations. Not so~$\rmE_l(\Ga_1(N))$, which can be decomposed into characters. More precisely, since~$\rmE_l(\Ga_1(N))$ consists of~$T$\nbd invariant functions, it descends to a representation of the abelian quotient~$\Ga_0(N) \slash \Ga_1(N)$. In classical terms, we have
\begin{gather*}
  \rmE_l(\Ga_1(N))
=
  \bigoplus_{\chi \pmod{N}}
  \rmE_l(\Ga_0(N), \chi)
\tx{,}
\end{gather*}
which parallels the representation theoretic decomposition
\begin{gather*}
  \Ind_{\Ga_1(N)}^{\Ga_0(N)}\,
  \bbone
\cong
  \bigoplus_{\chi \pmod{N}}
  \chi
\tx{,}
\end{gather*}
where the direct sum in both equations runs over Dirichlet characters~$\chi$ and in the second one we view~$\chi$ as a representations as in~\eqref{eq:def:dirichlet_type}.

When formulating Cohen's algorithm in the language of Theorem~\ref{thm:main_algorithm_preparation}, we have to compute
\begin{gather*}
  \bigoplus_{\chi'_1, \chi'_2 \pmod{N}}
  \rmH^0\big( \chi^{\prime\,\vee}_1 \otimes \chi^{\prime\,\vee}_2 \otimes \chi \big)
\tx{.}
\end{gather*}
Each of these spaces is nonzero if and only if~$\chi = \chi'_1 \chi'_2$. The most time-consuming part of Algorithm~\ref{alg:main_algorithm}, the computation of invariants, is hence reduce to a factorization in a finite, commutative group. This explains the superior performance of Cohen's implementation.

\subsection{Products of Weil-type Eisenstein and theta series}%
\label{ssec:product_weil_type_eisenstein_theta_series}

Based on his thesis~\cite{williams-2018a}, Williams provided an algorithm to calculate a basis of~$\rmM_k(\rho)$ where~$\rho$ is a Weil representation. Recall from, for instance~\cite{scheithauer-2009}, that the Weil representation is associated with a finite quadratic module~$(F,q)$. If the signature of~$(F,q)$ is even, the associated Weil representation is a representation of~$\SL{2}(\ZZ)$, and in general of the metaplectic group, which we do not introduce here. One can circumvent the limitation to Weil representation by the fact that every congruence type is a subrepresentation of a suitable Weil representation. The dimension of this enveloping Weil representation, however, might be very large.

As opposed to the work of Borisov--Gunnells, the work by Williams a priori rather features Jacobi forms~\cite{eichler-zagier-1985} and not products of Eisenstein series. To make the analogy to Algorithm~\ref{alg:main_algorithm} clear, we need some preparation. We write~$\rho_m$ for the Weil representation associated with the finite quadratic module~$(\ZZ \slash 2m \ZZ, x \mto x^2 \slash 4m)$. Recall that classical Jacobi forms are functions in two variables~$\tau$ and~$z$ on~$\HS \times \CC$. We will need the connection between Jacobi forms of index~$m$ and type~$\rho$, and vector-valued modular forms of type~$\rho \otimes \rho_m^\vee$ via the theta decomposition. It features the Jacobi theta series~$\theta_{m,l}$ for~$l \,\pmod{2m}$. A Jacobi form of index~$m$ can be written as
\begin{gather}
\label{eq:williams_specialization}
  \phi(\tau, z)
=
  \sum_{l \pmod{2m}}
  f_l(\tau)\, \theta_{m,l}(\tau,z)
\tx{,}\quad
  (f_l)_{l \pmod{2m}}
\in
  \rmM_{k-\frac{1}{2}}\big( \rho \otimes \rho_m^\vee \big)
\tx{.}
\end{gather}
We write~$\rmE^\rmJ_{k,m}(\rho)$ for the space of Jacobi Eisenstein series of weight~$k$, index~$m$, and type~$\rho$. For our purpose, it is important to record that the theta decomposition maps it to usual Eisenstein series.

Williams's algorithm is founded on the combination of two facts: First, the specialization of Jacobi Eisenstein series to~$z = 0$ yields ``Poincar\'e square-series''. Second, specific Dirichlet convolutions of Poincar\'e square--series yield usual Poincar\'e series. Since Poincar\'e series span spaces of modular forms, one obtains
\begin{gather}
  \rmM_k(\rho)
=
  \sum_{m = 1}^\infty \rmE^\rmJ_{k,m}(\rho) \big|_{z = 0}
\tx{,}
\end{gather}
An upper bound on which Eisenstein series are required on the right hand side can be deduced from a Sturm bound for~$\rmM_k(\rho)$ and the connection to Poincar\'e series. The resulting algorithm is very efficient provided that the Fourier expansion of Jacobi Eisenstein series can be computed at the required level of generality. Besides the relation between Poincar\'e square-series and Poincar\'e series, the need for these Fourier expansions is the source of the restriction of Williams's algorithm to Weil representations~$\rho$.

The theta decomposition allows us to recognize the specialization of Jacobi Eisenstein series in~\eqref{eq:williams_specialization} as a sum of products of Eisenstein series of weight~$k - \frac{1}{2}$ and theta series of weight~$\frac{1}{2}$. Since these theta series appear as residues of Eisenstein series, we allows ourselves to write~$\cE_{\frac{1}{2}}(\rho_m)$ for the space of modular forms (of half-integral weight) spanned by the~$\theta_{m,l}$, $l\pmod{2m}$, and~$\rmE_{\frac{1}{2}}(\rho_m)$ for the corresponding space of vector-valued modular forms. A weaker form of Williams's results in the spirit of Theorem~\ref{thm:main_algorithm_preparation} is
\begin{gather*}
  \rmM_k(\rho)
\;\cong\;
  \sum_{m = 1}^\infty
  \rmH^0\big(
  (\cE_{k-\frac{1}{2}}(\rho \otimes \rho_m^\vee) \cdot \cE_{\frac{1}{2}}(\rho_m))
  \otimes
  \rho
  \big)
\tx{.}
\end{gather*}
When evaluating this isomorphism, the invariants that occur are
\begin{gather}
\label{eq:williams_H0}
  \rmH^0\big( \rho^\vee \otimes ( \rho \otimes \rho_m^\vee ) \otimes \rho_m \big)
\tx{.}
\end{gather}

But Williams successfully avoids the calculation of invariants altogether. He thus skips the most time consuming part of Algorithm~\ref{alg:main_algorithm}. Specifically, his results employ canonical elements~$\id_m$ in~\eqref{eq:williams_H0}. We view them as homomorphisms from modular forms of arithmetic type~$\rho \otimes \rho_m^\vee \otimes \rho_m$ to modular forms of type~$\rho$. Then Williams's result can be rephrased as:
\begin{gather}
  \rmM_k(\rho)
=
  \sum_m
  \id_m\big(
  \rmE_{k-\frac{1}{2}}(\rho \otimes \rho_m^\vee) \otimes \rmE_{\frac{1}{2}}(\rho_m)	
  \big)
\tx{.}
\end{gather}
This explains the superior performance of Williams's implementation.

\ifbool{nobiblatex}{%
  \bibliographystyle{alpha}%
  \bibliography{bibliography.bib}%
}{%
  \vspace{1.5\baselineskip}
  \renewcommand{\baselinestretch}{.8}
  \Needspace*{4em}
  \printbibliography[heading=none]%
}

\Needspace*{3\baselineskip}
\noindent
\rule{\textwidth}{0.15em}

{\noindent\small
Chalmers tekniska högskola och Göteborgs Universitet,
Institutionen för Matematiska vetenskaper,
SE-412 96 Göteborg, Sweden\\
E-mail: \url{tobmag@chalmers.se}
}\vspace{.5\baselineskip}

{\noindent\small
Chalmers tekniska högskola och Göteborgs Universitet,
Institutionen för Matematiska vetenskaper,
SE-412 96 Göteborg, Sweden\\
E-mail: \url{martin@raum-brothers.eu}\\
Homepage: \url{http://martin.raum-brothers.eu}
}%

\end{document}

%% file: packages.tex
\usepackage{etoolbox}
\usepackage{ifdraft}
\usepackage{iftex}

\ifluatex
\usepackage[utf8]{luainputenc}
\usepackage[T1]{fontenc}
\fi

\usepackage[english]{babel}

\usepackage{lmodern}
\usepackage{fourier}

\usepackage{amsfonts}
\usepackage{mathrsfs} %
\usepackage{dsfont}

\usepackage{amssymb}

\usepackage{amsmath}
\usepackage{mathtools}
\usepackage[thmmarks, amsmath, amsthm]{ntheorem}
\usepackage{nccmath}    %

\usepackage[draft=false]{scrlayer-scrpage}

\usepackage{wrapfig}
\usepackage{needspace}

\providebool{nobiblatex}
\boolfalse{nobiblatex}
\ifbool{nobiblatex}{%
}{%
  \usepackage[backend=biber,style=numeric-comp,giveninits,url=false,sortcites=true,maxbibnames=99]{biblatex}
  \addbibresource{bibliography.bib}
}

\usepackage[final,
            pdfborder={0 0 0}, colorlinks=true,
            linkcolor=BrickRed, citecolor=ForestGreen, urlcolor=RoyalBlue]{hyperref}

\usepackage[cmyk,dvipsnames]{xcolor}

\ifdraft{%
\usepackage[textsize=scriptsize,colorinlistoftodos]{todonotes}
\usepackage{lineno}
\usepackage{scrtime}
}{}

\usepackage[linesnumbered,ruled]{algorithm2e}
\usepackage{booktabs}
\usepackage[inline]{enumitem}
\usepackage{float}
\usepackage{graphicx}
\usepackage{lettrine}
\usepackage[final]{listings}
\usepackage{multicol}
\usepackage{multirow}
\usepackage{tikz}
\usetikzlibrary{arrows,arrows.meta,calc,matrix,positioning}
\usepackage{url}

%% file: layout.tex
\KOMAoptions{DIV=10}

\ifdraft{%
\linenumbers
\newcommand*\patchAmsMathEnvironmentForLineno[1]{%
  \expandafter\let\csname old#1\expandafter\endcsname\csname #1\endcsname
  \expandafter\let\csname oldend#1\expandafter\endcsname\csname end#1\endcsname
  \renewenvironment{#1}%
     {\linenomath\csname old#1\endcsname}%
     {\csname oldend#1\endcsname\endlinenomath}}%
\newcommand*\patchBothAmsMathEnvironmentsForLineno[1]{%
  \patchAmsMathEnvironmentForLineno{#1}%
  \patchAmsMathEnvironmentForLineno{#1*}}%
\AtBeginDocument{%
\patchBothAmsMathEnvironmentsForLineno{equation}%
\patchBothAmsMathEnvironmentsForLineno{align}%
\patchBothAmsMathEnvironmentsForLineno{flalign}%
\patchBothAmsMathEnvironmentsForLineno{alignat}%
\patchBothAmsMathEnvironmentsForLineno{gather}%
\patchBothAmsMathEnvironmentsForLineno{multline}%
}}

\clearmainofpairofpagestyles
\automark[section]{subsection}

\renewcommand{\subsectionmark}[1]{}

\lohead{{\small\normalfont \headertitle}}
\rohead{{\small \headerauthors}}

\RedeclareSectionCommand[%
  font=\Large\sffamily\bfseries,%
  beforeskip=1\baselineskip,%
  afterskip=0.5\baselineskip,%
  indent=0em,%
  afterindent=false%
  ]{section}

\RedeclareSectionCommands[%
  font=\normalfont\bfseries,%
  beforeskip=3pt,%
  afterskip=-1em%
  ]{subsection,subsubsection}

\RedeclareSectionCommands[%
  font=\normalfont\itshape,%
  beforeskip=1pt,%
  afterskip=-1em,%
  indent=0pt,%
  ]{paragraph}

\ifbool{nobiblatex}{}{%
  \AtEveryBibitem{\clearfield{doi}}
  \AtEveryBibitem{\clearfield{isbn}}
  \AtEveryBibitem{\clearfield{issn}}
  \AtEveryBibitem{\clearfield{pages}}
  \AtEveryBibitem{\clearlist{language}}

  \setlength\bibitemsep{3pt}
  \renewcommand*{\bibfont}{\footnotesize}
  \renewbibmacro*{in:}{}
  \DeclareFieldFormat
    [article,inbook,incollection,inproceedings,patent,thesis,unpublished]
    {title}{#1}
}

\newenvironment{enumeratearabic*}{
\begin{enumerate*}[label=(\arabic*)] %
}{
\end{enumerate*}
}

\newenvironment{enumerateroman*}{
\begin{enumerate*}[label=(\roman*)] %
}{
\end{enumerate*}
}

\numberwithin{equation}{section}

\theoremnumbering{arabic}
\newtheorem{theoremcounter}{theoremcounter}[section]
\theoremnumbering{Alph}
\newtheorem{maintheoremcounter}{maintheoremcounter}

\theoremstyle{plain}

\newtheorem{corollary}[theoremcounter]{Corollary}
\newtheorem{lemma}[theoremcounter]{Lemma}

\newtheorem{proposition}[theoremcounter]{Proposition}
\newtheorem{theorem}[theoremcounter]{Theorem}

\theoremstyle{plain}

\newtheorem{maintheorem}[maintheoremcounter]{Theorem}

\theoremstyle{definition}

\newtheorem{definition}[theoremcounter]{Definition}

\theoremstyle{remark}

\newtheorem{remark}[theoremcounter]{Remark}

\theoremstyle{nonumberremark}

\lstdefinelanguage{julia}
{%
  keywordsprefix=\@,
  morekeywords={%
    error,throw,assert,
    begin,end,
    if,elseif,else,
    try,catch,
    while,for,do,break,continue,
    function,return,where,
    type,mutable,abstract,const,
    in,
    macro,
    quote,
    let,
    ccall,
    using,module,import,export,importall,
    local,global,
    Any,Nothing,None,Nullable,
  },
  sensitive=true,
  morecomment=[l]{\#},
  morestring=[b]",
}

\SetStartEndCondition{ }{}{}%

\SetKwProg{Pydef}{def}{\string:}{}
\SetKwComment{tcc}{\#\ }{}

\SetKwFor{For}{for}{\string:}{}%
\SetKwFor{While}{while}{:}{fintq}%
\SetKwIF{If}{ElseIf}{Else}{if}{:}{elif}{else:}{}%

\SetKwFunction{Range}{range}%
\SetKw{KwIn}{in}
\SetKw{KwNot}{not}
\SetKw{KwOr}{or}
\SetKw{KwInlineIf}{if}

\AlgoDontDisplayBlockMarkers\SetAlgoNoEnd\SetAlgoNoLine%

%% file: shortcuts.tex
\newcommand{\tx}{\text}
\newcommand{\thdash}{\nbd th}
\newcommand{\nbd}{\nobreakdash-\hspace{0pt}}

\makeatletter
\newcommand{\writelabel}[1]{#1\def\@currentlabel{#1}}
\makeatother

\newcommand{\minwidthmathbox}[2]{%
  \mathmakebox[{\ifdim#1<\width\width\else#1\fi}]{#2}%
}

\newcommand{\tbf}{\bfseries}

\newcommand{\bbF}{\ensuremath{\mathbb{F}}}

\newcommand{\bbone}{\ensuremath{\mathds{1}}}

\newcommand{\cB}{\ensuremath{\mathcal{B}}}

\newcommand{\cE}{\ensuremath{\mathcal{E}}}

\newcommand{\cQ}{\ensuremath{\mathcal{Q}}}

\newcommand{\frake}{\ensuremath{\mathfrak{e}}}

\newcommand{\rmE}{\ensuremath{\mathrm{E}}}

\newcommand{\rmH}{\ensuremath{\mathrm{H}}}
\newcommand{\rmJ}{\ensuremath{\mathrm{J}}}

\newcommand{\rmL}{\ensuremath{\mathrm{L}}}
\newcommand{\rmM}{\ensuremath{\mathrm{M}}}

\newcommand{\rmO}{\ensuremath{\mathrm{O}}}

\newcommand{\rmS}{\ensuremath{\mathrm{S}}}
\newcommand{\rmT}{\ensuremath{\mathrm{T}}}

\newcommand{\llbrkt}{\llbracket}
\newcommand{\rrbrkt}{\rrbracket}

\newcommand{\mrelspace}[1]{\mathrel{\mspace{#1}}}

\let\rightarroworig\rightarrow
\renewcommand{\rightarrow}
  {\protect\relbar\mrelspace{-9.7mu}\rightarroworig}
\renewcommand{\hookrightarrow}
  {\protect\lhook\mrelspace{-3.1mu}\relbar\mrelspace{-11.9mu}\rightarroworig}
\renewcommand{\twoheadrightarrow}
  {\protect\rightarroworig\mrelspace{-15mu}\rightarroworig}

\let\leftarroworig\leftarrow
\renewcommand{\leftarrow}
  {\protect\leftarroworig\mrelspace{-9.7mu}\relbar}

\renewcommand{\longrightarrow}
  {\protect\relbar\mrelspace{-3.2mu}\relbar\mrelspace{-9.5mu}\rightarroworig}

\newcommand{\longtwoheadrightarrow}
  {\protect\relbar\mrelspace{-3mu}\rightarroworig\mrelspace{-15mu}\rightarroworig}

\newcommand{\xtwoheadrightarrow}[1]
  {\protect\xrightarrow{#1}\mrelspace{-15mu}\rightarroworig}

\newcommand{\xlongtwoheadrightarrow}[1]
  {\xtwoheadrightarrow{\minwidthmathbox{1.256em}{#1}}}

\newcommand{\ra}{\rightarrow}
\newcommand{\hra}{\hookrightarrow}
\newcommand{\thra}{\twoheadrightarrow}

\newcommand{\lra}{\longrightarrow}

\newcommand{\lthra}{\longtwoheadrightarrow}

\newcommand{\mto}{\mapsto}
\newcommand{\lmto}{\longmapsto}

\renewcommand{\Re}{\mathrm{Re}}
\renewcommand{\Im}{\mathrm{Im}}

\newcommand{\isdiv}{\mathop{\mid}}

\renewcommand{\pmod}[1]{\;(\mathrm{mod}\, #1)}

\newcommand{\Hom}{\operatorname{Hom}}
\newcommand{\id}{\mathrm{id}}

\newenvironment{psmatrix}{\left(\begin{smallmatrix}}{\end{smallmatrix}\right)}
\newcommand{\Mat}[1]{\operatorname{Mat}_{#1}}

\newcommand{\T}{\rmT}

\newcommand{\linspan}{\operatorname{span}}

\newcommand{\ZZ}{\ensuremath{\mathbb{Z}}}
\newcommand{\QQ}{\ensuremath{\mathbb{Q}}}
\newcommand{\RR}{\ensuremath{\mathbb{R}}}
\newcommand{\CC}{\ensuremath{\mathbb{C}}}

\newcommand{\GL}[1]{\ensuremath{\mathrm{GL}_{#1}}}

\newcommand{\SL}[1]{\ensuremath{\mathrm{SL}_{#1}}}

\newcommand{\U}[1]{\ensuremath{\mathrm{U}_{#1}}}

\newcommand{\std}{\ensuremath{\mathrm{std}}}

\newcommand{\sym}{\ensuremath{\mathrm{sym}}}

\newcommand{\HS}{\ensuremath{\mathbb{H}}}

%% file: shortcuts_this.tex
\newcommand{\Ga}{\Gamma}
\newcommand{\ga}{\gamma}
\newcommand{\rank}{\mathrm{rank}}
\newcommand{\isotypGaN}[1]{\isotypGasub{#1}{N}}
\newcommand{\isotypGasub}[2]{\big[ e\big( \mfrac{#1}{#2} \big) \big]}
\newcommand{\rmFE}{\mathrm{FE}}
\newcommand{\rmfe}{\mathrm{fe}}
\newcommand{\QQab}{\QQ^{\mathrm{ab}}}
\newcommand{\Stab}{\mathrm{Stab}}
\newcommand{\Ind}{\mathrm{Ind}}
\newcommand{\Res}{\mathrm{Res}}
\newcommand{\ahol}{\mathrm{ahol}}
\newcommand{\lcm}{\mathrm{lcm}}